\documentclass[a4paper,11pt]{article}
\usepackage{color}
\usepackage{amsmath,amsfonts,amsthm, amssymb,xspace}
\definecolor{blue}{rgb}{0.00,0.00,1.00}
\definecolor{red}{rgb}{1.00,0.00,0.00}

\newcommand{\Z}{{\mathbb Z}}
\newcommand{\N}{{\mathbb N}}

\newcommand{\F}{{\mathbb F}}
\newcommand{\cA}{{\cal A}}
\newcommand{\cB}{{\cal B}}
\newcommand{\cS}{{\cal S}}
\newcommand{\cL}{{\cal L}}
\newcommand{\ctL}{{\cal L}^0}
\newcommand{\cM}{{\cal M}}
\newcommand{\ctM}{{\cal M}^0}
\newcommand{\ainv}{{ \hbox{\sc a}}}
\newcommand{\binv}{{\hbox{\sc b}}}
\newcommand{\cinv}{{\hbox{\sc c}}}
\newcommand{\dinv}{{\hbox{\sc d}}}
\newcommand{\einv}{{\hbox{\sc e}}}
\newcommand{\sL}{{ \hbox{\sc l}}}
\newcommand{\sU}{{ \hbox{\sc u}}}
\newcommand{\cX}{{\cal X}}
\newcommand{\cY}{{\cal Y}}

\newcommand{\sys}{\rm{sys}}
\newcommand{\Res}{{\rm{Res}}}
\newcommand{\Cay}{{\cal G}}
\newcommand{\dist}{{\rm{dist}}}
\newcommand{\GAP}{{\sf GAP}}
\newcommand{\KBMAG}{{\sf KBMAG}}
\newtheorem{theorem}{Theorem}[section] 
\newtheorem{lemma}[theorem]{Lemma}     
\newtheorem{proposition}[theorem]{Proposition}
\newtheorem{corollary}[theorem]{Corollary}
\def\margin_comment#1{\marginpar{\sffamily{\small #1\par}\normalfont}}
\newenvironment{mylist}{\begin{list}{}{
\setlength{\parskip}{0mm}
\setlength{\topsep}{2mm}
\setlength{\parsep}{0mm}
\setlength{\itemsep}{0.5mm}
\setlength{\labelwidth}{7mm}
\setlength{\labelsep}{3mm}
\setlength{\itemindent}{0mm}
\setlength{\leftmargin}{12mm}
\setlength{\listparindent}{6mm}
}}{\end{list}}
\title{Biautomatic structures in systolic Artin groups}
\author{Derek F. Holt and Sarah Rees,\\
Mathematics Institute, University of Warwick, Coventry CV47AL, UK,\\
School of Mathematics, Statistics and Physics, University of Newcastle, Newcastle NE1 7RU, UK}
\date{6th July 2018}
\begin{document}
\maketitle

\begin{abstract}
We examine the construction of Huang and Osajda that was used in their proof of
the biautomaticity of Artin groups of almost large type. We describe a slightly
simpler variant of that biautomatic structure, with explicit descriptions
of a few small examples, and we examine some of the properties of the structure.
We explain how the construction can be programmed within the {\sf GAP system}.
\end{abstract}

Mathematics subject classification (2010): {20F36 (primary), 20F10, 20F65 (secondary)}

Keywords: Artin groups, word problem, automatic groups, biautomatic groups, systolic complexes, systolic groups.

\section{Introduction}
The work reported in this article was motivated by a recent article of
Huang and Osajda \cite{HuangOsajda} proving the biautomaticity of
a class of Artin groups, namely those of \emph{almost large type},
which contains all Artin groups of large type.
Previously biautomaticity
had been proved by Brady and McCammond for many, but not all,
Artin groups of large type, specifically for those of extra-large type \cite{Peifer},
all 3-generated examples, and all others for which the associated Coxeter
diagram can be oriented in such a way as to exclude certain oriented
subdiagrams \cite{BradyMcCammond}.
Biautomaticity had also been proved for all Artin groups of finite
type by Charney \cite{Charney, Charney2}; for these groups (which are Garside groups), symmetric, geodesic biautomatic
structures were found over the Garside generators.

We were particularly interested in the results of \cite{HuangOsajda} since
we already knew that all Artin
groups of almost large type 
(in fact all in the slightly larger class of sufficiently large type) had 
shortlex automatic structures over their standard generating sets 
\cite{SLartin,SLartin2}; but these structures are not in general biautomatic.
Our methods were combinatorial, based on rewriting.

The methods used in \cite{BradyMcCammond} and \cite{Charney,Charney2} 
to prove biautomaticity for those two types of Artin groups 
are quite distinct from each other.
Brady and McCammond's 
approach is geometric, relying on the construction of a piecewise 
Euclidean non-positively curved 2-complex on which the Artin group
acts discretely and fixed point freely, followed by the application of results
of Gersten and Short \cite{GerstenShort1,GerstenShort2}.
But Charney's approach uses the Garside structure of Artin groups of finite type.

Huang and Osajda's article \cite{HuangOsajda} proves biautomaticity of an Artin group 
of almost large type via the construction of a systolic complex on which the 
group acts simplicially, properly discontinuously and cocompactly;
hence by definition, the group is systolic. 
The biautomaticity of these Artin groups then follows immediately
from \cite[Theorem E]{JS}, that all systolic groups are biautomatic. 
We observe that, in fact, the universal covers of the 2-complexes of 
\cite{BradyMcCammond} are systolic, although the construction of those 2-complexes does 
not seem to follow the same pattern as that of the systolic complexes in
\cite{HuangOsajda}; hence in some sense the results of \cite{HuangOsajda}
generalise those of \cite{BradyMcCammond}.
Our aim in this work has been to examine the construction of \cite{HuangOsajda} and see
what we can learn from it. 

To a large extent our approach has been experimental,
through computation and examination of examples.
In particular, we have written programs in $\GAP$ that construct the biautomatic
structures arising from the systolic complexes that are acted on by
Artin groups of almost large type in which all edge labels in the
associated Coxeter diagram are either at most $4$ or equal to infinity.

In order to describe and analyse the construction of \cite{HuangOsajda} we need a
variety of background material on automatic and biautomatic groups, on Artin 
groups, and on systolic complexes, and we have tried to give sufficient detail to make this article accessible to a range of readers.

After this introductory section, Section~\ref{sec:gpthy_intro}
introduces the concepts and notations we shall use from group theory.
In Section~\ref{sec:intro_biauto_artin} we define 
automatic and biautomatic structure for groups, and introduce Artin groups,
while Section~\ref{sec:gp_actions} explains briefly how group 
actions may be used to build such structures. 

Section~\ref{sec:systolic} introduces the concepts and notation of systolic
complexes and systolic groups, and is based on \cite{JS}. 
The basic concepts are introduced in Section~\ref{sec:systolic_intro}, then Section
~\ref{sec:dirgeo} explains the concepts of directed and allowable geodesics 
that are vital to the construction of biautomatic
structures in systolic groups.
Section~\ref{sec:biaut} is devoted to a description of the construction of
biautomatic structures for systolic groups, again with reference to \cite{JS}.
Section~\ref{sec:gensets} describes 
generating sets $\cB$ and $\cA$ over which biautomatic structures will be
defined; descriptions of those biautomatic structures,
and of the automata that define them, are given in the subsequent sections.
In fact we introduce two slightly different biautomatic structures,
over the two alphabets $\cB$ and $\cA$. The language $\cL'$ over $\cB$
is the structure of \cite{JS}, but we prefer in this article to work
with the languages $\cL$ and $\ctL$, over $\cA$,
which are closely related to $\cL'$.

Huang and Osajda's construction of a systolic complex $\cX$ for an Artin group $G$ of almost large 
type is described briefly at the beginning of Section~\ref{sec:systolic_artin}.
The complex $\cX$ is formed out of the Cayley graph $\Cay$ for $G$ by
replacing subgraphs $\Cay_{ij}$, corresponding to 2-generator 
subgroups $G_{ij}$ of $G$, by systolic complexes $\cX_{ij}$, and the 
construction of those
subcomplexes is described in Section~\ref{sec:subcomplexes_Xij}.
We illustrate the construction of $\cX$ by computing some examples in Section~\ref{sec:Artin_examples}. We describe the complexes for the
right-angled Artin groups $\Z^2$, $\Z^2 * \Z$ and $\F_2 \times \Z$ in 
Sections~\ref{sec:dihartZ2},~\ref{sec:Z2freeZ},~\ref{sec:F2timesZ},
and the 3-string braid group $G(A_2)$ in Section~\ref{sec:dihartA2}.

Section~\ref{sec:properties} is devoted to the identification of some
properties found in the biautomatic structure of any Artin group of almost large type.
Here, we may assume that the Artin group is non-free; for a free group in its natural presentation,
the associated systolic complex is simply the Cayley graph over the standard
generating set.
Corollary~\ref{cor:cA_Garside} shows that, given a sensible selection of
the orbit representatives upon which the language $\ctL$ depends, the generating
set $\cA$ (as well as the related set $\cB$) consists precisely
of the union of inverse closed Garside generating sets for the 2-generated Artin subgroups $G_{ij}$ of $G$,
together with a symbol representing the identity element.
Subsequent results examine properties of the language $\ctL$ over $\cA$.
It is clear from our examples that $\ctL$ is not, in general,
geodesic over $\cA$.
But we prove in Proposition~\ref{prop:wordlength} that the $\cA$-length of a word in $\ctL$ is no greater than its geodesic length over the
standard Artin generating set $X$,
and no less than $1/\max\{2,M-2\}$ of that geodesic length, where $M$ is the maximum of all finite edge labels $m_{ij}$. 
We observe too, from our examples, that the language $\ctL$ is not in general
prefix closed, but from a result about extensions of directed geodesics 
that we prove in Proposition~\ref{prop:extend_dirgeo}
we can deduce 
Corollary~\ref{cor:prefix}, which associates to each prefix of a word 
in $\ctL$ a word in $\ctL$ differing from it by a bounded amount at its end.

Finally, Section~\ref{sec:computer_code} describes our development of algorithms
that construct systolic complexes and then biautomatic structures for those systolic Artin groups that we were
able to handle.

\section{Concepts and notation from group theory}
\label{sec:gpthy_intro}
\subsection{Introducing biautomatic groups and Artin groups}
\label{sec:intro_biauto_artin}
All the groups considered in this article will be finitely presented.
Let $G=\langle X \mid R \rangle$ be such a group. We define its {\em Cayley
graph} $\Cay=\Cay(G,X)$ to be the graph with vertex set $G$ and, for each
$x \in X$, directed edges labelled $x$ and $x^{-1}$ from $g$ to $gx$ and from
$gx$ to $g$, respectively. These two directed edges have the same
underlying undirected edge.

A finitely presented group $G=\langle X \mid R \rangle$
is {\em automatic} if there exists
\begin{mylist}
\item[(1)]  a regular set $L$ of words over $X$ 
providing a complete set of representatives of the elements of $G$,
and 
\item[(2)] an associated integer $k$,
\end{mylist}
satisfying the following condition:
\begin{quote}
whenever words $u,v$ in $L$ represent elements
$g,h \in G$ satisfying $ga=_G h$ with $a \in X \cup X^{-1} \cup \{1 \}$, 
and $\gamma_1(u),\gamma_1(v)$ are 
the paths in $\Cay(G,X)$ that are traced out by
$u,v$ from the identity vertex, 
then $\gamma_1(u)$ and $\gamma_1(v)$ fellow travel at distance $k$.
\end{quote}

By definition, a set of words is {\em regular} if it is the language of a finite state automaton.
We say that two paths in a graph {\em fellow travel} at distance $k$ if, for any $i$,
the distance in the graph between the vertices at distance $i$ from the initial
vertex of each path is at most $k$.

An automatic group $G$ is {\em biautomatic} if there exist $L,k$ as above satisfying
the following additional condition:
\begin{quote}
whenever words $u,v \in L$ represent elements
$g$ and $h \in G$ satisfying $ah=_G g$ with $a \in X \cup X^{-1}$, 
and $\gamma_1(u)$ and $\gamma_a(v)$
are the paths in $\Cay(G,X)$ that are traced out by
$u$ from the identity vertex and by $v$ from $a$, 
then $\gamma_1(u)$ and $\gamma_a(v)$ fellow travel at distance $k$.
\end{quote}
We refer to \cite{ECHLPT,HRRbook} for more detail.
But we note in particular that the properties of automaticity and biautomaticity are
independent of the choice of finite generating set; that is, if a group G has 
an automatic or biautomatic structure over some finite generating set $X$ 
then it has a corresponding structure over any other finite generating set.
We note also that any hyperbolic group is biautomatic;
over any generating set, the set of all geodesic words is the language of a
biautomatic structure \cite[Theorem 3.4.5]{ECHLPT}.

The language $L$ of an automatic structure for a group $G$ is called {\em symmetric} if whenever $w \in L$ then $w^{-1} \in L$;
when $L$ is symmetric then the 
structure satisfies both fellow traveller conditions and so $G$ is biautomatic.
$L$ is called {\em geodesic} if every word in $L$ is a minimal length 
representative of the group element it represents. $L$ is called {\em prefix closed} if whenever $w \in L$ then every prefix of $w$ is also in $L$. 

In this article we consider biautomatic structures for certain types
of  Artin groups.  The standard presentation for an Artin group over its
standard generating set $X =\{a_1,\ldots, a_n\}$ is as
\[ \langle a_1,\ldots ,a_n \mid {}_{m_{ij}}(a_i,a_j)=
{}_{m_{ji}}(a_j,a_i)\quad\hbox{\rm for each}\quad i \neq j \rangle, \]
where the integers $m_{ij}$ are the entries in a Coxeter matrix
(a symmetric $n \times n$ matrix $(m_{ij})$ with entries in
$\N \cup \{\infty\}$, where $m_{ii}=1$, and $m_{ij} \geq 2$, $\forall
i \neq j$), and where
for generators $a,a'$ and $m \in \N$ we define ${}_m(a,a')$ to be the
word that is a product of $m$ alternating $a$'s and $a'$'s and
starts with $a$.
Adding the relations $a_i^2=1$ to those for the Artin group defines the associated Coxeter group,
which is more commonly presented as
\[ \langle a_1,\ldots ,a_n \mid (a_ia_j)^{m_{ij}}=1\quad\hbox{\rm for each}\quad i,j \rangle. \]
The standard presentation of an Artin group can be described using an associated
{\em Coxeter diagram} $\Gamma$, with $n$ nodes, where nodes corresponding to $a_i$ 
and $a_j$ are joined by an edge labelled by $m_{ij}$.

An Artin group $G(\Gamma)$ is said to be of {\em spherical} or {\em finite} type if the associated
Coxeter group is finite, of {\em dihedral type} if the associated
Coxeter group is dihedral (or, equivalently, the standard generator set has two elements),
and of {\em large type} if $m_{ij} \geq 3$ for all $i \neq j$.
The group $G(\Gamma)$ is defined in
\cite{HuangOsajda} to have {\em almost large type} if 
(1) in any triangle of edges of $\Gamma$ either at least one edge is labelled $\infty$ or no edge has label 2, and 
(2) in any square of edges of $\Gamma$ either at least one edge is labelled $\infty$ or at most one edge has label 2. 
So, in particular, all Artin groups of large type are of almost large type.

Artin groups of various types are proved automatic \cite{Charney,Peifer,BradyMcCammond,VanWyk,HermillerMeier,SLartin,SLartin2}, but before the results of
\cite{HuangOsajda} biautomatic structures were only known to exist for
right angled Artin groups \cite{HermillerMeier}, those of finite type,
and many (but not all) those of large type. Automatic structures
for which the associated language $L$ is a set of shortlex geodesics over the
standard generating set were already known for all almost large Artin groups
and, slightly more generally, for all \emph{sufficiently large} Artin groups
\cite{SLartin2}).  

\subsection{Using group actions to derive automatic structures}
\label{sec:gp_actions}
We recall a basic tenet of geometric group theory
commonly referred to as the Milnor-\u{S}varc lemma:
if a group $G=G(X)$ has a `nice' (properly discontinuous and compact) discrete, 
isometric action on a metric space $\cX$,
then its Cayley graph $\Cay=\Cay(G,X)$ is quasi-isometric to $\cX$.
Specifically, for any vertex $x_0 \in \cX$, there is a quasi-isometry from $\Cay$ to $\cX$ that maps the identity vertex of $\Cay$ to $x_0$, and each vertex
$g$ of $\Cay$ to the image $gx_0$ in $\cX$ of $x_0$ under the action of $g$. 

This is at the basis of results such as those that follow, of Gersten and Short, and of Niblo and Reeves.

\begin{theorem}[Gersten\&Short, 1990,1991 \cite{GerstenShort1,GerstenShort2}]
The fundamental groups of piecewise Euclidean 2-complexes of types $A_1 \times A_1$, $A_2$, $B_2$ and $G_2$
(corresponding to tesselations of the Euclidean plane by squares, equilateral
triangles, and triangles with angles $(\pi/2,\pi/4,\pi/4)$, $(\pi/2,\pi/3,\pi/6)$) are
biautomatic.
\end{theorem}
This result is stated in terms of automatic rather than biautomatic structures in \cite[Theorem1]{GerstenShort2},
but it is made clear in the introduction of that paper that the construction gives biautomatic structures.

It follows from Gersten and Short's results that a group acting discretely and fixed point freely on 
a complex of any of the types listed must be biautomatic. 

\begin{theorem}[Niblo\&Reeves 1998, Theorem 5.3 of \cite{NibloReeves_biaut}]
If $\cX$ is a simply connected and non-positively curved cube complex, and $G$ acts
effectively, cellularly, properly discontinuously and cocompactly on $\cX$,
then $G$ is biautomatic.
\end{theorem}

A cube complex is a metric polyhedral complex in which each cell is 
isometric to a Euclidean cube and the glueing maps are isometries. 
It is proved in \cite{NibloReeves_coxeter} that any
finite rank Coxeter group $G$ acts properly discontinuously by isometries 
on an appropriate cube complex; in the many cases where the action is cocompact, the Coxeter group
is proved biautomatic.

The results of \cite{GerstenShort1,GerstenShort2} are used in the proof of the
following result.

\begin{theorem}[Brady\&McCammond, 2000 \cite{BradyMcCammond}]
Various Artin groups of large type, including all that are 3-generated,
act appropriately on
piecewise Euclidean non-positively curved
(i.e. locally CAT(0)) 
2-complexes of types $A_2$ or $B_2$, and hence are biautomatic.
\end{theorem}

The 2-complexes constructed in the proof of this theorem are
simplicial,
made by attaching angles and lengths to  presentation complexes for 
non-standard presentations for the Artin groups, in which all relators have 
length 3.
It can be seen that their universal covers are systolic \cite{JS}, 
as defined below.

\section{Systolic complexes}
\label{sec:systolic}
\subsection{Basic concepts}
\label{sec:systolic_intro}
We need the definition of a systolic complex, and related notation,
which we shall use throughout this article.
In this we generally follow \cite{JS}.
Suppose that $\cX$ is a simplicial complex, and $\cS=\cS(\cX)$ its set of
simplices (the vertices of its `barycentric subdivision').
We denote by $\cX_1$
the 1-skeleton of $\cX$, 
that is, the graph on its vertices whose edges are the 1-cells.
The set $\cS$ itself has the structure of a simplicial complex, whose 
simplices are sets of simplices of $\cX$ that can be ordered into an ascending
chain of simplices that are related by inclusion; in particular
its set $E(\cS)$ of  1-cells (edges) is the set of pairs
$\{\rho,\sigma\}$ of distinct simplices of $\cX$ for
which either $\sigma \subset \rho$ or $\rho \subset \sigma$.
We denote the 1-skeleton of $\cS$ by $\cS_1$.

We define a {\em cycle} in $\cX$ to be a 1-dimensional subcomplex of $\cX$
consisting of the vertices and edges of a closed path in $\cX_1$
in which all edges are distinct and only the first and last vertices coincide;
the length of the cycle is defined to be the number of 1-cells in it.
We call a  cycle a {\em full cycle} if it is
full as a subcomplex, that is, if any simplex spanned by 
a subset of its vertices is within the cycle. This means that no two 
non-adjacent vertices within the cycle are joined by an edge (or rather, 
there are no chords between non-adjacent vertices of the cycle).
Note that, since a cycle is 1-dimensional, a full cycle in a simply
connected complex cannot have length 3.
We define the {\em systole} of $\cX$, $\sys(\cX)$, to be the minimum length of
a full cycle.

We write $\cX_\sigma$ for the {\em link} in $\cX$ of a simplex $\sigma$ 
(the set of simplices $\tau$ such that $\tau$ is disjoint from $\sigma$, and 
such that $\tau$ and $\sigma$
span a simplex $\sigma*\tau$ of $\cX$), and for any subcomplex $\cY$ of $\cX$ 
containing $\sigma$, we write $\Res(\sigma,\cY)$ for the {\em residue} of 
$\sigma$ in $\cY$ (the union of all closed simplices that contain
$\sigma$; this is sometimes called the 
{\em closed star} of $\sigma$ in $\cY$).
We define the $1$-ball $B_1(\sigma,\cY)$ of $\sigma$ in $\cY$
to be the union of all closed simplices within $\cY$ that intersect 
$\sigma$ non-trivially.


We say that $\cX$ is {\em $\kappa$-large} if $\sys(\cX) \geq \kappa$ and also 
$\sys(\cX_\sigma)\geq \kappa$ for every simplex $\sigma$, {\em locally $\kappa$-large}
if the residue of every simplex of $\cX$ is $\kappa$-large, {\em $\kappa$-systolic}
if it is connected, simply connected and locally $\kappa$-large.
We call a group $\kappa$-systolic if it acts properly discontinuously
and cocompactly by automorphisms on a $\kappa$-systolic complex. 
(Proper discontinuity means that stablisers are finite.)
We abbreviate $6$-systolic as {\em systolic}.
By \cite{JS} $\kappa$-systolic complexes with $\kappa\geq 6$ are $\kappa$-large
(but this isn't true for $\kappa=4,5$). And (locally) $\kappa$-large implies
(locally) $m$-large for $\kappa \geq m$.

We note the following results from \cite{JS}
\begin{theorem}[Januszkiewicz\&Swiatkowski, Theorem A \cite{JS}]
The 1-skeleton of a 
$7$-systolic  complex is hyperbolic; any $7$-systolic group is hyperbolic.
\end{theorem}

\begin{theorem}[Januszkiewicz\&Swiatkowski, Theorem E \cite{JS}]
Any systolic group is biautomatic.
\end{theorem}

In order to prove a group $G$  biautomatic, we need to define a regular language
$L$ of words (over some selected generating set) that maps onto the group 
and has appropriate fellow travelling properties.  
When $G$ acts appropriately on a systolic complex $\cX$, we
fix a vertex (0-cell) $x_0$ of $\cX$.
Then for each 
element $g \in G$  we need to select a path (or paths) within the complex $\cX$ from $x_0$ 
to $gx_0$, and then associate with each such path a word over 
the selected generated set 
that represents $g$.
It turns out that directed geodesics between pairs of vertices $x_0$, $gx_0$ have some very good
properties; we use the Milnor-\u{S}varc lemma to translate those properties into properties for the words in our selected language.

\subsection{Directed and allowable geodesics}
\label{sec:dirgeo}
Suppose that $\cX$ is a locally 6-large (but not necessarily systolic) 
simplicial complex.
We call a sequence of simplices $\sigma_0,\ldots,\sigma_n$ from $\cS(\cX)$ 
a {\em directed geodesic} from $\sigma_0$ to $\sigma_n$ if
\begin{mylist}
\item[(1)] for all $i$ with $0 \leq i \leq n-1$, $\sigma_i$ and $\sigma_{i+1}$ are disjoint and span a simplex $\sigma_i*\sigma_{i+1}$ of $\cX$, and
\item[(2)] for all $i$ with $0 \leq i \leq n-2$, $\Res(\sigma_i, \cX_{\sigma_{i+1}}) \cap B_1(\sigma_{i+2},\cX_{\sigma_{i+1}}) = \emptyset$.
\end{mylist}

We note that\\ 

\begin{lemma}
\label{lem:subpath_dir_geo}
\begin{mylist}
\item[(1)] Every subpath of a directed geodesic is a directed geodesic,
\item[(2)] If $\sigma_0,\ldots,\sigma_n$ is a directed geodesic, and $v$ is a vertex of the simplex $\sigma_n$, then $\sigma_0,\ldots,\sigma_{n-1},v$ is a directed geodesic.
\end{mylist}
\end{lemma}
\begin{proof}
It follows straight from the definition that a sequence of simplices 
is a directed geodesic if and only if the same is true of every subsequence of length at most 3, and so (1) is immediate.

(2) is immediate once we observe that if $\sigma_n \in \cX_{\sigma_{n-1}}$ and
$v$ is a  vertex of $\sigma_n$, then 
$B_1(v,\cX_{\sigma_{n-1}})\subseteq B_1(\sigma_n,\cX_{\sigma_{n-1}})$.
\end{proof} 
Note that the reverse of a directed geodesic need not be a directed geodesic.

If $\cX$ is systolic, then for any pair of vertices $v,w$ of $\cX$,
there is a unique directed geodesic  $\gamma(v,w)=\sigma_0,\sigma_1,\ldots,\sigma_n$ 
from $v=\sigma_0$ to $w=\sigma_n$ \cite[Lemma 9.7]{JS},
and it is geodesic, that is $d_{\cX_1}(v,w)=n$ \cite[Corollary 9.8]{JS}.
That directed geodesic is the {\em projection ray} from $v$ to $w$ (as defined in \cite{JS});
it is proved in \cite[Lemma 9.3]{JS} that every projection ray in a systolic 
complex is a directed geodesic, and in \cite[Proposition 9.6]{JS} that every directed geodesic $\sigma_0,\ldots,\sigma_n$ is a projection
ray on its final simplex $\sigma_n$.

Given vertices $v,w$, we call an {\em allowable geodesic} \cite[Section 11]{JS}
an infinite sequence
of vertices  
$u_0,u_1,\ldots$ such that if $\sigma_0,\sigma_1,\ldots,\sigma_n$ is the
directed geodesic from $v=\sigma_0$ to $w=\sigma_n$, then $u_0=v,u_i=w$ for $i \geq n$, and $u_i \in \sigma_i$ for $0<i<n$.
Such a sequence $u_0,u_1,\ldots,u_n$ is a geodesic path in the 1-skeleton 
of $\cX$ \cite[Fact 11.1]{JS}.
It is proved in \cite[Prop 11.2]{JS} that a pair of allowable geodesics joining
(respectively) vertex $v$ to vertex $w$ and vertex $p$ to vertex $q$ must
fellow travel at distance at most
$3 \max\{\dist_{\cX_1}(v,p),\dist_{\cX_1}(w,q)\}+1$.

From a directed geodesic $\gamma=\gamma(v,w)$ as above, we can define an associated
{\em polygonal path} 
\[\hat{\gamma}= \hat{\gamma}(v,w)=\sigma_0,\sigma_0*\sigma_1,\sigma_1,\sigma_1*\sigma_2,\sigma_2, \cdots,\sigma_{n-1},\sigma_{n-1}*\sigma_n,\sigma_n,\]
and note that any two adjacent simplices in a polygonal path are related by (alternately) inclusion or reverse inclusion.
So a polygonal path is a path within the graph $\cS_1$, 
whereas consecutive simplices on a directed 
geodesic  are at distance 2 as vertices of $\cS_1$.

\section{Biautomatic structures for $G$}
\label{sec:biaut}
\subsection{Alternate generating sets $\cB$ and $\cA$ for $G$}
\label{sec:gensets}
Now suppose that $G$ is a systolic group associated with the systolic complex $\cX$.
Let $K$ be the set of orbits of the action of $G$ on $\cS(\cX)$,
and let $V_0$ be a set of representatives of the orbits.
Choose $v_0 \in V_0$ to be a 0-cell. The 0-cell $v_0$ is a natural 
choice for $x_0$, which we use as we apply the Milnor-\u{S}varc lemma to embed 
$\Cay$ in $\cX$ as described in Section~\ref{sec:gp_actions}.

For $\sigma \in \cS(\cX)$,
let $\bar{\sigma}$ denote its representative in $V_0$. For 
$\sigma \in \cS(\cX)$, define $\Lambda_\sigma:= \{ g \in G: \sigma=g\bar{\sigma}\}$, the `set of {\em labels} of $\sigma$';
note that, since we have $v_0 \in V_0$, we have $\Lambda_{v_0} = G_{v_0}$.
Then for 
any pair $(\sigma,\tau)$ of simplices from $\cS(\cX)$, define $\Lambda_{\sigma,\tau}:= \Lambda_\sigma^{-1}\Lambda_\tau$, the `set of labels of $(\sigma,\tau)$';
note that if stabilisers are all trivial, then the sets $\Lambda_\sigma$ and $\Lambda_{\sigma,\tau}$ are all singletons. 
Note also that, for any $g \in G$, $\Lambda_{g\sigma,g\tau}=\Lambda_{\sigma,\tau}$, that $G_{v_0}\Lambda_{v_0,\sigma}=\Lambda_{v_0,\sigma}$,
and that for all $\rho,\sigma,\tau$, 
$\Lambda_{\rho,\sigma}\Lambda_{\sigma,\tau}=\Lambda_{\rho,\tau}$
We define an alphabet $\cB$, as in \cite[Lemma 14.3]{JS},
to be a (finite) set of symbols representing the union of $G_{v_0}$ and all the sets $\Lambda_{\sigma,\tau}$
for which $\sigma$ and $\tau$ are simplices related by inclusion or reverse inclusion.
Note that $\cB$ includes a symbol representing the identity element $1$.

We also define an alphabet
$\cA$,  a set of symbols representing the union of $G_{v_0}$
with all the sets $\Lambda_{\rho,\tau}$ 
for which $\rho,\tau$ are disjoint simplices that together span a simplex.

We will define biautomatic structures for $G$ over each of the two alphabets.
The first, which we call $\cL'$, over the alphabet $\cB$, and associated with polygonal paths, is the structure described in \cite[Page 49]{JS}.
We prefer the second, $\cL$,  over the alphabet $\cA$, and associated with 
directed geodesics. It is closely associated with $\cL'$, but the words in it 
(over $\cA$) are generally shorter than those in $\cL'$ (over $\cB$).
When the subgroup $G_{v_0}$ is trivial, we can simplify $\cL$ further
by deleting the first symbol in every word;
we shall call that simplified language $\ctL$.

\subsection{A structure over $\cB$}
\label{sec:biautB}

The following biautomatic structure over $\cB$ is (essentially) described both
in \cite[Pages 49-51]{JS} and in
\cite{Swiatkowski}. We have made some small adjustments after
noticing some differences between the two descriptions; we believe that there is an error in the
suggestion in \cite{JS} that any occurrences of the symbol $1$ in a word should
be deleted from it before it is included in the language; for
the language of words shortened
in this way would not satisfy the fellow traveller conditions.

Given $g \in G$,
let $\gamma= \gamma(v_0,gv_0)$
be the unique directed geodesic from $v_0$ to $gv_0$, and 
$\hat{\gamma}=\hat{\gamma}(v_0,gv_0)$ the corresponding polygonal path.
Write $\gamma=\sigma_0=v_0,\ldots,\sigma_i,\ldots,\sigma_n=gv_0$, $\hat{\gamma}=\sigma'_0=v_0,\ldots,\sigma'_j,\ldots,\sigma'_{2n}=gv_0$ (so that 
$\sigma'_{2i}=\sigma_i$).
Following \cite{JS}, we associate to $\hat{\gamma}$ 
all sequences of elements $g_j \in G$
($j=0,\ldots,2n$) with $g_j \in \Lambda_{\sigma'_j}$ for each $j$
( so $g_0 \in G_{v_0}$) and $g_{2n}=g$,
and then 
all words $g_0\nu_1\cdots \nu_{2n}$ over $\cB$, where
 $\nu_j \in \Lambda_{\sigma'_{j-1},\sigma'_j}$ ($j=1,\ldots,2n$) is defined by
$\nu_j := g_{j-1}^{-1}g_j$.

We see that the word $g_0\nu_1\cdots \nu_{2n}$
represents the element 
$g$, and that its proper prefixes represent $g_0,g_1,\ldots,g_{2n-1}$. 
We define $\cL'$ to be the set of all such words $g_0\nu_1\cdots \nu_{2n}$. 
If all stabilisers are trivial then $\cL'$ contains a unique representative of each element of $G$, but otherwise some elements will admit more than one representative.

That $\cL'$ defines a biautomatic structure for $G$ is proved in \cite{JS}.
Regularity of $\cL'$ is verified by explicit construction of an automaton that
recognises $\cL'$. Fellow travelling of appropriate pairs of words in $\cL'$ 
is inherited via the Milnor-\u{S}varc lemma from the fellow travelling of 
related pairs of allowable geodesics.

\subsection{A structure over $\cA$}
\label{sec:biautA}

We can form a biautomatic structure for $G$ over the alphabet $\cA$, directly
from the directed geodesics, as follows.

Given $g \in G$,
again let $\gamma= \gamma(v_0,gv_0)=\sigma_0=v_0,\sigma_1,\ldots,\sigma_n=gv_0$
be the unique directed geodesic from $v_0$ to $gv_0$. 
Now, associate to $\gamma$ 
all sequences of elements $h_i$
($i=0,\ldots,n$) with $h_i \in \Lambda_{\sigma_i}$ for each $i$ (so $h_0 \in G_{v_0}$), and $h_n=g$, and then all words
$h_0\mu_1\cdots \mu_{n}$ over $\cA$, where $\mu_i \in \Lambda_{\sigma_{i-1},\sigma_i}$ ($i=1,\ldots,n$) is defined by
$\mu_i := h_{i-1}^{-1}h_i$; 
we define the language $\cL$ over $\cA$ to be the set of all such words.

Again, we see that the word $h_0\mu_1\cdots \mu_n$
represents the element 
$g$, and that its proper prefixes represent $h_0,h_1,\ldots,h_{n-1}$. 
Clearly, each of the generators $\mu_j$ is equal in $G$ to a word of length 2 over $\cB$, and so there is a natural map from $\cL$ to $\cL'$, in which generators of $\cA$ are replaced by words of length 2 over $\cB$.

Because of the relationship between $\cL'$ and $\cL$, it is clear that
$\cL$ is also a biautomatic structure.

In addition,
 we note that the word $w=h_0\mu_1\cdots\mu_n$ representing $g$ has length 
$n+1$, while the directed geodesic $\gamma(v_0,gv_0)$ has length $n$, the same 
as the length of an allowable geodesic from $v_0$ to $gv_0$. 

\subsection{Constructing the automaton that recognises $\cL'$}
\label{sec:construct1}
The description of the automaton $\cM'$ that recognises $\cL'$ is taken from
\cite[Page 50]{JS}, with a small amount of adjustment.

We need the notation introduced in Section~\ref{sec:gensets}.
Recall that $K$ denotes the set of orbits  $G\sigma$
of the simplices $\sigma$ in the set $\cS=\cS(\cX)$, under the action of $G$. 
Recall also that $\cS$ itself has the structure of a simplicial complex.
We denote by $E(K)$ the set of orbits of $G$ on the set $E(\cS)$ of
edges of $\cS$;  its elements are sets 
$G\{\rho,\sigma\} := \{\{g\rho,g\sigma\}: g \in G \}$,
for which $\{\rho,\sigma\}\in E(\cS)$\}.
Now, as above let $\bar{\rho}, \bar{\sigma}$ be the representatives in the set
$V_0$ of the orbits
$G\rho, G\sigma$, and choose $\lambda_\rho \in \Lambda_\rho, \lambda_\sigma \in \Lambda_\sigma$; so $\rho = \lambda_\rho \bar{\rho}, 
\sigma = \lambda_\sigma\bar{\sigma}$. Then, although it is not necessarily true that 
$\{\bar{\rho},\bar{\sigma}\} \in E(\cS)$, we have $\{\lambda_\rho \bar{\rho},\lambda_\sigma\bar{\sigma}\} \in E(\cS)$, and it follows that
$\{\lambda^{-1}\bar{\rho},\bar{\sigma}\},\{\bar{\rho},\lambda \bar{\sigma}\} \in E(\cS)$ for all $\lambda \in \Lambda_{\rho,\sigma}$.

In the automaton $\cM'$, we 
have a unique start state, which we label $v_0$,
and for each $h \in G_{v_0}$ we have a state $(v_0,h)$.
In addition, for each element $G\{\rho,\sigma\}$ of $E(K)$,
and each $\lambda \in \Lambda_{\rho,\sigma}$, we define states
$(G(\rho,\sigma),\lambda)$.
For ease of notation, we refer to the state
$(G(\rho,\sigma),\lambda)$ as $(\rho,\sigma,\lambda)$,
but of course, for any
$g \in G$, the triple $(g\rho,g\sigma,\lambda)$ represents that same state.
The accept states are $(v_0,1)$ and all states of the form $(\sigma,v_0,\lambda)=(g\sigma,gv_0,\lambda)$.
The arrows are as follows:
\begin{mylist}
\item[(1)] an arrow labelled $h$ from $v_0$ to $(v_0,h)$, for each $h \in G_{v_0}=\Lambda_{v_0}$,
\item[(2)] for each $\sigma \in \cS$ with $v_0 \subset \sigma$,
and each $\lambda \in \Lambda_{v_0,\sigma}$,
an arrow labelled $\lambda$ from $(v_0,h)$ to the state
$(v_0,\sigma,\lambda)$, 
\item[(3)] an arrow labelled $\mu$ from state $(\rho,\sigma,\lambda)$ to state $(\sigma,\tau,\mu)$,
provided that:
\begin{mylist}
\item[either (i)] $\lambda^{-1}\bar{\rho}$ and $\mu \bar{\tau}$ are disjoint and span $\bar{\sigma}$,
\item[or (ii)] $\bar{\sigma}$ is a proper face of both $\lambda^{-1}\bar{\rho}$ and $\mu \bar{\tau}$ and\\
$\Res((\lambda^{-1}\bar{\rho})_{\bar{\sigma}},\cX_{\bar{\sigma}}) \cap 
B_1((\mu\bar{\tau})_{\bar{\sigma}},\cX_{\bar{\sigma}}) = \emptyset.$
\end{mylist}
\end{mylist}

We note that all paths from the start state to accept states have odd length,
alternating after the start between states $(\rho,\sigma,\lambda)$ with $\rho$
contained in $\sigma$ and those with $\sigma$ contained in $\rho$.
We note that all paths of length 1 from the start state lead to states of the 
form $(v_0,h)$, all paths of odd length greater than 1 from the start state
lead to states of the form
$(\rho,\sigma,\lambda)$ with $\rho \supset \sigma$,
and all paths of even length greater than 0 from the start state lead to states
of the form
$(\rho,\sigma,\lambda)$ with $\rho \subset \sigma$.

\subsection{Constructing the automaton that recognises $\cL$}
\label{sec:construct2}
We modify the construction of $\cM'$ to construct an automaton $\cM$ that recognises $\cL$; we use just the start state of $\cM'$ and those states at odd distance from it in $\cM'$.
We label the unique start state $v_0$,
and have accepting states $(v_0,h)$, for each $h \in G_{v_0}$, as before.
Then for each element $G\{\rho,\sigma\}$ of $E(K)$ with $\rho \supset \sigma$,
and each $\lambda \in \Lambda_{\rho,\sigma}$, we define a state
$(\rho,\sigma,\lambda)$. 
As above, if $g \in G$, the triples $(\rho,\sigma,\lambda)$ and $(g\rho,g\sigma,\lambda)$ represent the same state.
As above, the accepting states are
$(v_0,1)$ and all states of the form $(\sigma,v_0,\lambda)=(g\sigma,gv_0,\lambda)$.
The arrows are as follows:
\begin{mylist}
\item[(1)] an arrow labelled $h$ from $v_0$ to $(v_0,h)$, for each $h \in G_{v_0}=\Lambda_{v_0}$,
\item[(2)] for each $\rho \in \cS$ with $v_0 \subset \rho$,
and $\lambda \in \Lambda_{v_0,\rho}$, where $\sigma \subset \rho$
and $\mu \in \Lambda_{\rho,\sigma}$, an edge labelled $\lambda\mu$ from 
$(v_0,h)$ to the state $(\rho,\sigma,\mu)$,
provided that $\lambda^{-1}v_0$ and $\mu\bar{\sigma}$ are disjoint and span
$\bar{\rho}$,
\item[(3)] an arrow labelled $\mu\nu$ 
from state $(\rho,\sigma,\lambda)$ to state $(\tau,\upsilon,\nu)$, 
where $\sigma\subset \tau$ and $\mu \in \Lambda_{\sigma,\tau}$,
provided that:
\begin{mylist}
\item[(i)] the simplices $\mu^{-1}\bar{\sigma}$ and $\nu \bar{\upsilon}$ are
disjoint and span $\bar{\tau}$;
\item[(ii)] the simplex $\bar{\sigma}$ is a proper face of both
$\lambda^{-1}\bar{\rho}$ and $\mu \bar{\tau}$; and
\item[(iii)] $\Res((\lambda^{-1}\bar{\rho})_{\bar{\sigma}},\cX_{\bar{\sigma}}) \cap 
B_1((\mu\bar{\tau})_{\bar{\sigma}},\cX_{\bar{\sigma}}) = \emptyset.$
\end{mylist}
\end{mylist}

We observe that, if the vertex stabilisers are trivial then the first letter 
in every element of $\cL$ is simply the letter $1$, representing
 the identity element. In this case we might prefer to delete it;
we shall denote by $\ctL$  the language we get from $\cL$
by this simple adjustment. Clearly the new language also defines a 
biautomatic structure. 
It is elementary to modify the description  of $\cL$ above,
to find an automaton that recognises $\ctL$, as follows:
we no longer need the state labelled $(v_0,1)$, or the type 1 transitions
from $v_0$ to $(v_0,1)$, the state $v_0$ becomes an accepting state,
and we replace the type 2 transitions by transitions on the same
letter, and with the same target, but with source $v_0$ rather than $(v_0,1)$.

\section{Building systolic complexes and biautomatic structures for Artin groups}
\label{sec:systolic_artin}
\subsection{The basic construction}
Let $G$ be an Artin group of almost large type, defined  in its natural 
presentation over the standard generating set $X$. 
The systolic complex $\cX$ of \cite{HuangOsajda} is built
out of the Cayley graph $\Cay$ for $G$ over $\cX$; the vertex $v_0$ is set to be the vertex of $\Cay$ that corresponds to the identity element.
For each (parabolic) subgroup $G_{ij} = \langle x_i,x_j \rangle$ of $G$ with $i \ne j$
(which is itself a 2-generated Artin group) and $m_{ij} \ne \infty$,
we build a systolic complex $\cX_{ij}$ out of the Cayley graph $\Cay_{ij}$ for
$G_{ij}$. This involves adjoining new vertices and edges and also
new higher dimensional simplices to $\Cay_{ij}$.
So $\Cay_{ij}$ embeds in $\cX_{ij}$, and the action of $G_{ij}$ 
on $\cX_{ij}$ is inherited from its action on $\Cay_{ij}$.
We describe this construction in more detail for the cases $m_{ij}=2,3$ in Sections~
\ref{sec:dihartZ2} and ~\ref{sec:dihartA2} below.

For each vertex $v$ of $\Cay$, there is a subgraph of $\Cay$
isomorphic to $\Cay_{ij}$ with base point $v$.  We construct the systolic
complex $\cX$ for $G$ by replacing each of these subgraphs for which $m_{ij}<\infty$
by the corresponding systolic complex $\cX_{ij}$. When we make this replacement
we leave unchanged those vertices and edges that already belong to  
$\Cay$. So, for each $v$ and each such pair $i,j$, we are adjoining a
collection of new vertices, edges and higher dimensional simplices to $\cX$.
The action of $G$ on $\Cay$ extends naturally to an action on $\cX$.

Since $G$ acts fixed point freely on $\Cay$, in this action of $G$ on 
the systolic complex $\cX$
all simplex stablilisers are trivial. So the subgroup $G_{v_0}$ is the identity
subgroup, and the sets $\Lambda_\sigma$ and $\Lambda_{\sigma,\rho}$ are all
singletons; we shall denote the unique element of $\Lambda_\sigma$ by
$\lambda_\sigma$, 
the unique element of $\Lambda_{\sigma,\tau}$ by $\lambda_{\sigma,\tau}$.

\subsection{The modification of $\Cay_{ij}$ that produces $\cX_{ij}$}
\label{sec:subcomplexes_Xij}
The process that transforms the Cayley graph $\Cay_{ij}$ for the 2-generator 
Artin group $G_{ij}$ into a systolic complex $\cX_{ij}$, when $m_{ij}< \infty$, is described in some 
detail in \cite[Section 3]{HuangOsajda}. We note only a few points here
that we need in order to identify features of the biautomatic structures.

Under the action of $G_{ij}$, the Cayley graph $\Cay_{ij}$ contains vertices
in a single orbit, and there are two orbits on undirected edges. These
correspond to the directed edges labelled $a_i^{\pm 1}$ and those labelled
$a_j^{\pm 1}$.

A circuit of length $2m_{ij}$ labelled by the word
that is the concatenation of $a_ia_j\cdots$ and the inverse of $a_ja_j\cdots$ 
(both of length $m_{ij}$) starts at each vertex; such a circuit is called a
{\em precell} in \cite{HuangOsajda}. 
We call the two vertices on the boundary of a precell $\Pi$  that are (respectively) the sources 
and targets of two directed 
edges (labelled $x_i$ and $x_j$) its {\em initial} and {\em terminal vertices},
and we call the two portions of the boundary that consist of paths running
from the 
initial vertex to the terminal vertex its two {\em half boundaries}. 
We note that if two precells have intersecting boundaries, then that
intersection consists either of a single vertex,
which is the initial vertex of one and the terminal vertex of the other,
or of a path that is a proper prefix of a half boundary of one and a proper
suffix of a half boundary of the other.

The first step of the process that transforms 
$\Cay_{ij}$ to $\cX_{ij}$ is the triangulation of each such precell by the 
addition of $m_{ij}-2$ new vertices, $5m_{ij}-9$ new edges, and $4m_{ij}-6$ new 
2-cells, as shown in Figure~\ref{fig:triangulation}.
In \cite{HuangOsajda} the new vertices are called {\em interior vertices},
and the vertices inherited from $\Cay_{ij}$ {\em real vertices};
we shall use the same convention here.
In fact we shall call both vertices and edges created during the triangulation 
of a precell $\Pi$
{\em interior vertices and edges of $\Pi$}.

\begin{figure}
\setlength{\unitlength}{0.8pt}
\begin{picture}(300,70)(-60,-40)
\put(-80,-5){initial vertex}
\put(0,0){\circle*{3}} 
\put(240,0){\circle*{3}}
\put(245,-5){terminal vertex}

\put(40,0){\circle*{3}} \put(80,0){\circle*{3}}
\put(85,0){\circle*{1}} \put(90,0){\circle*{1}} \put(95,0){\circle*{1}} 
\put(100,0){\circle*{1}} \put(105,0){\circle*{1}} \put(110,0){\circle*{1}}
\put(115,0){\circle*{1}} \put(120,0){\circle*{1}} \put(125,0){\circle*{1}} 
\put(130,0){\circle*{1}} \put(135,0){\circle*{1}} \put(140,0){\circle*{1}}
\put(145,0){\circle*{1}} \put(150,0){\circle*{1}} \put(155,0){\circle*{1}}
\put(160,0){\circle*{3}} \put(200,0){\circle*{3}} 

\put(0,0){\line(1,0){40}} \put(40,0){\line(1,0){40}}
\put(160,0){\line(1,0){40}} \put(200,0){\line(1,0){40}}

\put(80,35){half boundary}
\put(20,30){\circle*{3}} \put(60,30){\circle*{3}} \put(100,30){\circle*{3}}
\put(105,30){\circle*{1}} \put(110,30){\circle*{1}} \put(115,30){\circle*{1}} 
\put(120,30){\circle*{1}} \put(125,30){\circle*{1}} \put(130,30){\circle*{1}} 
\put(135,30){\circle*{1}}  
\put(140,30){\circle*{3}} \put(180,30){\circle*{3}} \put(220,30){\circle*{3}}

\put(0,0){\line(2,3){20}} \put(10,15){\vector(2,3){0}} 
\put(20,30){\line(1,0){40}} \put(60,30){\line(1,0){40}}
\put(40,30){\vector(1,0){0}} \put(45,30){\vector(1,0){0}} \put(80,30){\vector(1,0){0}}
\put(140,30){\line(1,0){40}} \put(180,30){\line(1,0){40}}
\put(160,30){\vector(1,0){0}} \put(165,30){\vector(1,0){0}} \put(200,30){\vector(1,0){0}}
\put(240,0){\line(-2,3){20}} \put(230,15){\vector(2,-3){0}} \put(234,9){\vector(2,-3){0}}

\put(40,0){\line(-2,3){20}} \put(40,0){\line(2,3){20}}
\put(80,0){\line(-2,3){20}} \put(80,0){\line(2,3){20}}
\put(160,0){\line(-2,3){20}} \put(160,0){\line(2,3){20}}
\put(200,0){\line(-2,3){20}} \put(200,0){\line(2,3){20}}

\put(80,-45){half boundary}
\put(20,-30){\circle*{3}} \put(60,-30){\circle*{3}} \put(100,-30){\circle*{3}}
\put(105,-30){\circle*{1}} \put(110,-30){\circle*{1}} \put(115,-30){\circle*{1}} 
\put(120,-30){\circle*{1}} \put(125,-30){\circle*{1}} \put(130,-30){\circle*{1}} 
\put(135,-30){\circle*{1}}  
\put(140,-30){\circle*{3}} \put(180,-30){\circle*{3}} \put(220,-30){\circle*{3}}

\put(0,0){\line(2,-3){20}} \put(10,-15){\vector(2,-3){0}} \put(14,-21){\vector(2,-3){0}}
\put(20,-30){\line(1,0){40}} \put(60,-30){\line(1,0){40}}
\put(40,-30){\vector(1,0){0}} \put(80,-30){\vector(1,0){0}} \put(85,-30){\vector(1,0){0}}
\put(140,-30){\line(1,0){40}} \put(180,-30){\line(1,0){40}}
\put(160,-30){\vector(1,0){0}} \put(200,-30){\vector(1,0){0}} \put(205,-30){\vector(1,0){0}}
\put(240,0){\line(-2,-3){20}} \put(230,-15){\vector(2,3){0}} 

\put(40,0){\line(-2,-3){20}} \put(40,0){\line(2,-3){20}}
\put(80,0){\line(-2,-3){20}} \put(80,0){\line(2,-3){20}}
\put(160,0){\line(-2,-3){20}} \put(160,0){\line(2,-3){20}}
\put(200,0){\line(-2,-3){20}} \put(200,0){\line(2,-3){20}}
\end{picture}
\caption{Triangulation of a precell}
\label{fig:triangulation}
\end{figure}
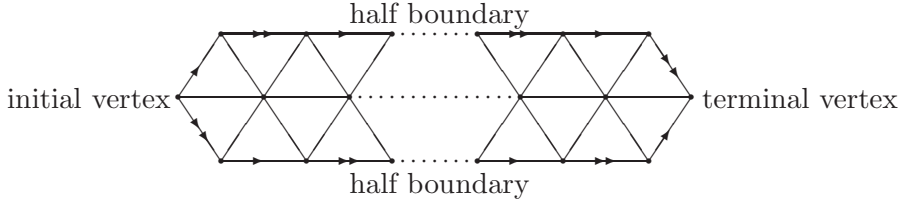

But this process produces some {\em bad links} when $m_{ij} > 2$.
At this stage, the link of each of the interior vertices is a hexagon,
but some of the real vertices have links that are 4-cycles or 5-cycles,
and so {\em bad}.
Figure~\ref{fig:badlink} shows what happens at the junction of two precells
when $m_{ij}=3$.

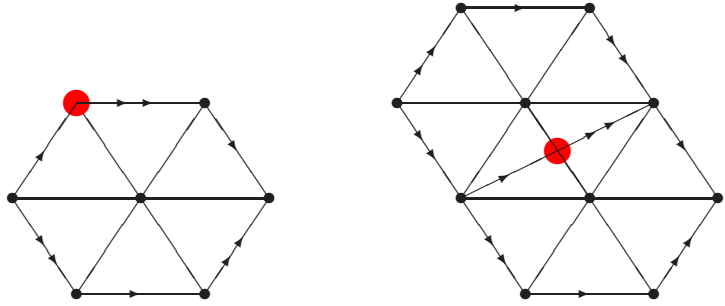
\begin{figure}
\setlength{\unitlength}{1.2pt}
\begin{picture}(400,110)(-60,-40)
\put(0,0){\circle*{3}} 

\put(40,0){\circle*{3}} \put(80,0){\circle*{3}}

\put(0,0){\line(1,0){40}} \put(40,0){\line(1,0){40}}
\put(40,0){\line(-2,3){20}} \put(40,0){\line(2,3){20}}
\put(40,0){\line(-2,-3){20}} \put(40,0){\line(2,-3){20}}

{\color{red}
\put(20,30){\circle*{8}}} 
\put(60,30){\circle*{3}} 

\put(0,0){\line(2,3){20}} \put(10,15){\vector(2,3){0}} 
\put(20,30){\line(1,0){40}} \put(36,30){\vector(1,0){0}} \put(44,30){\vector(1,0){0}}
\put(80,0){\line(-2,3){20}} \put(70,15){\vector(2,-3){0}} 

\put(20,-30){\circle*{3}} \put(60,-30){\circle*{3}} 

\put(0,0){\line(2,-3){20}} \put(10,-15){\vector(2,-3){0}} \put(14,-21){\vector(2,-3){0}}
\put(20,-30){\line(1,0){40}} \put(40,-30){\vector(1,0){0}} 
\put(80,0){\line(-2,-3){20}} \put(68,-18){\vector(2,3){0}} \put(72,-12){\vector(2,3){0}} 


\put(140,0){\circle*{3}} \put(220,0){\circle*{3}}

\put(180,0){\circle*{3}} 

\put(140,0){\line(1,0){40}} \put(180,0){\line(1,0){40}}
\put(180,0){\line(-2,3){20}} \put(180,0){\line(2,3){20}}
\put(180,0){\line(-2,-3){20}} \put(180,0){\line(2,-3){20}}

{\color{red}
\put(170,15){\circle*{8}}} 
\put(200,30){\circle*{3}} 

\put(140,0){\line(2,1){30}} \put(155,7.5){\vector(2,1){0}} 
\put(170,15){\line(2,1){30}} \put(182,21){\vector(2,1){0}} \put(188,24){\vector(2,1){0}}
\put(220,0){\line(-2,3){20}} \put(210,15){\vector(2,-3){0}} 

\put(160,-30){\circle*{3}} \put(200,-30){\circle*{3}} 

\put(140,0){\line(2,-3){20}} \put(150,-15){\vector(2,-3){0}} \put(154,-21){\vector(2,-3){0}}
\put(160,-30){\line(1,0){40}} \put(180,-30){\vector(1,0){0}} 
\put(220,0){\line(-2,-3){20}} \put(208,-18){\vector(2,3){0}} \put(212,-12){\vector(2,3){0}}

\put(120,30){\circle*{3}} 

\put(140,60){\circle*{3}} \put(180,60){\circle*{3}} 
\put(120,30){\line(2,3){20}} \put(128,42){\vector(2,3){0}} \put(132,48){\vector(2,3){0}}
\put(140,60){\line(1,0){40}} \put(160,60){\vector(1,0){0}} 
\put(180,60){\line(2,-3){20}} \put(188,48){\vector(2,-3){0}} 
\put(192,42){\vector(2,-3){0}} 

\put(160,30){\circle*{3}} 

\put(120,30){\line(1,0){40}} \put(160,30){\line(1,0){40}}
\put(160,30){\line(2,3){20}} \put(160,30){\line(2,-3){20}}
\put(160,30){\line(-2,3){20}} \put(160,30){\line(-2,-3){20}}

\put(120,30){\line(2,-3){20}} \put(128,18){\vector(2,-3){0}} \put(132,12){\vector(2,-3){0}}
\end{picture}
\caption{Triangulating the Cayley graph produces bad links at some vertices}
\label{fig:badlink}
\end{figure}

Those bad links are corrected by the addition of new edges, and then the attachment of higher dimensional simplices as necessary, in order to ensure that $\cX_{ij}$
continues to be {\em flag} 
(that is, any set of pairwise incident vertices forms a simplex).

This whole process of addition of edges and higher dimensional simplices that
terminates in the systolic complex $\cX_{ij}$ is described in detail in 
\cite[Section 3.2]{HuangOsajda}.
The following is clear from that description; 
part (i) is stated at the end of the second paragraph, and part (ii)
is clear from the description of the addition of edges in zigzag patterns.\\

\begin{lemma}
\label{lem:Xij_newedges}
\begin{mylist}
\item[(i)]
Each edge of $\cX_{ij} \setminus \Cay_{ij}$ is either an interior edge of a
single precell, or it joins two interior vertices within two distinct
overlapping precells $\Pi,\Pi'$ whose boundaries intersect in a path $\pi$ of
at least two edges, where $\pi$ contains the initial vertex of one of
$\Pi,\Pi'$, and the terminal vertex of the other. 
\item[(ii)] If $\Pi,\Pi'$ are overlapping precells, and  
$u_1,u_2,u_3$ are interior vertices with $u_1 \in \Pi$, $u_2,u_3 \in \Pi'$,
and such that $\{u_1,u_2\}$ and $\{u_1,u_3\}$ are edges,
then $\{u_2,u_3\}$ is an edge within $\Pi'$.
\end{mylist}
\end{lemma}

\begin{lemma}
\label{lem:pathlength}
Suppose that $G$ is non-free, and
that $v,v'$ are real vertices of $\cX$ at distance $n$ within the graph $\cX_1$.
Then $v,v'$ are at distance at most $n\max(2,M-2)$ within $\Cay$,
where $M := \max_{i,j}(\{ m_{ij} : m_{ij} \ne \infty \})$.
\end{lemma}
\begin{proof}
Let $v,v_1,\ldots,v_n=v'$ be the sequence of vertices of $\cX$ on a path
of length $n$ in $\cX_1$. It follows from the triangle inequality that
it is sufficient to prove the result when all vertices $v_i$ with $1 \le i < n$
(i.e those that are strictly between $v$ and $v'$) are interior.

If $n=1$, then either $v$ and $v'$ are joined by an edge in $\cX_1$, or
they are joined by an internal edge in a precell of a subcomplex $\cX_{ij}$
with $m_{ij}=2$. In either case $d_{\Cay}(v,v') \le 2$, and the result holds.

If $n>1$, then there are interior vertices between $v$ and $v'$ on the path,
and the path must lie within an $\cX_{ij}$ subcomplex, for some $i,j$ 
with finite $m_{ij}>2$. Let $m := m_{ij}$.

If $n=2$, then the edges $\{v,v_1\}$ and $\{v_1,v'\}$ are interior edges 
of a precell $\Pi$ of which $v,v'$ are boundary vertices. In that case
$d_\Cay(v,v') \leq m $ and so, since $m \leq \max(4,2(m-2))$,
the result holds. 

So now suppose that $n>2$.  Then each of the interior vertices $v_i$ 
is in a unique precell $\Pi_i$ and, if $\Pi_i\neq \Pi_{i+1}$,
then the boundaries of $\Pi_i$ and $\Pi_{i+1}$ intersect in a path containing
at least two edges. 

Let $u_i,w_i$ be the initial and terminal vertices, respectively, of the
precell $\Pi_i$.
Then $v$ is on the boundary of $\Pi_1$, so $d(v,u_1)+d(v,w_1) = m$ and
similarly $d(v',u_{n-1})+d(v',w_{n-1}) = m$. Hence either
$d(v,u_1) + d(v',u_{n-1}) \le m$ or $d(v,w_1) + d(v',w_{n-1}) \le m$
(or both). In the first case, we define $v'_i:=u_i$ for each $i=1,\ldots,n-1$,
and in the second case $v'_i:=w_i$ for each $i=1,\ldots,n-1$.
So in either case we have  $d_{\Cay}(v'_1,v)+d_{\Cay}(v'_{n-1},v')\leq m$.
If $v'_i \neq v'_{i+1}$ then, since $\Pi_i$ and $\Pi_{i+1}$ overlap,
sharing at least two edges of their 
boundaries, we must have $d_\Cay(v'_i,v'_{i+1}) \leq m-2$.
We deduce that 
\begin{eqnarray*}
d_\Cay(v,v') &\leq& d_\Cay(v,v'_1) + \sum_{i=1}^{n-2}d_\Cay(v'_i,v'_{i+1}) +
           d_\Cay(v'_{n-1},v')\\
&\leq& m + (n-2)(m-2) \leq n\max(2,m-2).\end{eqnarray*} 
as required.
\end{proof}

\section{Examples of Artin group complexes}
\label{sec:Artin_examples}
\subsection{The dihedral Artin group $\Z^2=\langle a,b \mid ab=ba \rangle $}
\label{sec:dihartZ2}
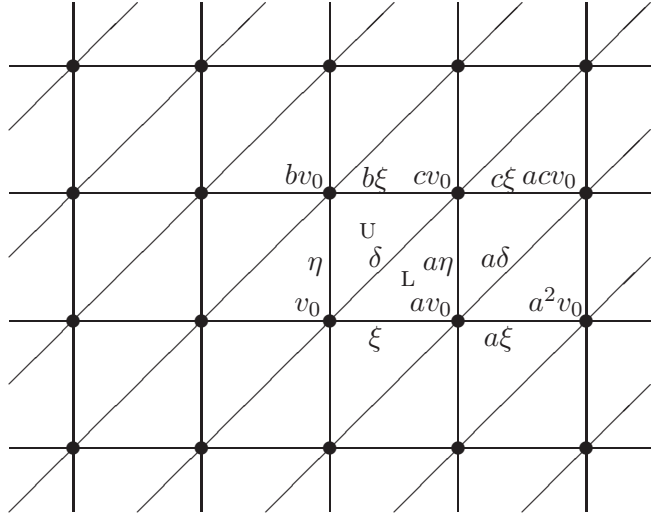
\begin{figure}
\setlength{\unitlength}{1.2pt}
\begin{picture}(260,160)(-120,-60)
\put(-100,-40){\line(1,0){200}}
\put(-80,-40){\circle*{4}} \put(-40,-40){\circle*{4}}
\put(0,-40){\circle*{4}} \put(40,-40){\circle*{4}} \put(80,-40){\circle*{4}}
\put(-100,0){\line(1,0){200}}
\put(-80,0){\circle*{4}} \put(-40,0){\circle*{4}}
\put(0,0){\circle*{4}} \put(40,0){\circle*{4}} \put(80,0){\circle*{4}}
\put(-100,40){\line(1,0){200}}
\put(-80,40){\circle*{4}} \put(-40,40){\circle*{4}}
\put(0,40){\circle*{4}} \put(40,40){\circle*{4}} \put(80,40){\circle*{4}}
\put(-100,80){\line(1,0){200}}
\put(-80,80){\circle*{4}} \put(-40,80){\circle*{4}}
\put(0,80){\circle*{4}} \put(40,80){\circle*{4}} \put(80,80){\circle*{4}}
\put(-80,-60){\line(0,1){160}} \put(-40,-60){\line(0,1){160}}
\put(0,-60){\line(0,1){160}}
\put(40,-60){\line(0,1){160}} \put(80,-60){\line(0,1){160}}
\put(-60,-60){\line(1,1){160}} 
\put(-100,-60){\line(1,1){160}} \put(-20,-60){\line(1,1){120}}
\put(-100,-20){\line(1,1){120}} \put(20,-60){\line(1,1){80}}
\put(-100,20){\line(1,1){80}} \put(60,-60){\line(1,1){40}}
\put(-100,60){\line(1,1){40}} 
\put(-11,3){$v_0$} \put(25,3){$av_0$} \put(62,3){$a^2v_0$} 
\put(-14,43){$bv_0$} \put(26,43){$cv_0$} \put(60,43){$acv_0$} 
\put(-7,16){$\eta$}
\put(29,16){$a\eta$} 
\put(12,17){$\delta$}
\put(47,17){$a\delta$} 
\put(22,11){$\sL$}
\put(9,26){$\sU$}
\put(12,-8){$\xi$} \put(48,-8){$a\xi$} 
\put(10,42){$b\xi$} \put(50,42){$c\xi$} 
\end{picture}
\caption{Systolic complex for $\Z^2$}
\label{fig:Z2_complex}
\end{figure}
Here the systolic complex $\cX$ is formed from the tesselation of the plane by 
the integer lattice, with each of the squares subdivided into two right-angled 
triangles (2-cells) by a diagonal (an interior edge) running bottom left to top right.
The generator $a$ translates one unit to the right, and $b$ one unit upwards.
So the vertices are all $(x,y)$ with $x,y \in \Z$, the edges are all pairs
$\{(x,y),(x+1,y)\}$, $\{(x,y),(x,y+1)\}$ and $\{(x,y),(x+1,y+1)\}$ 
and the 2-cells are all triples $\{(x,y),(x,y+1),(x+1,y+1)\}$
and $\{(x,y),(x+1,y),(x+1,y+1)\}$.
There are six orbits of the action of $G$ on $\cX$, and
$V_0 := \{ v_0,\xi,\eta,\delta,\sU,\sL\}$ is a set of orbit representatives,
where $v_0=(0,0)$ is the single vertex;
$\xi=\{(0,0),(1,0)\}$, $\eta= \{(0,0),(0,1)\}$, and $\delta = \{(0,0),(1,1)\}$
are edges parallel to the $x$-axis, the $y$-axis and the diagonal;
and $\sU = \{ (0,0),(0,1),(1,1)\}$ and $\sL = \{ (0,0),(1,0),(1,1)\} $
are upper and lower triangular 2-cells, as indicated in Figure~\ref{fig:Z2_complex}.

We have $\cB = \cA= \{ 1,a,b,d,\ainv,\binv,\dinv\}$, where we use 
the symbol $1$ to represent the identity element, $d$ to represent the product $ab$, and $\ainv,\binv,\dinv$ to 
denote the inverses of $a,b,d$ respectively.

In order to construct $\ctL$ we need to identify the directed geodesics.
We note that the simplices in a directed geodesic $\sigma_0,\ldots,\sigma_n$ can
only be of dimensions 0 or 1.

The link of an
edge (1-dimensional simplex) is a pair of vertices. 
So if $\sigma_{i+1}$ is an edge, then $\sigma_i$ and $\sigma_{i+2}$ must both be
vertices.

On the other hand, the link $\cX_u$ of a vertex (0-dimensional simplex) $u$ is 
the set of vertices and edges of a hexagon.
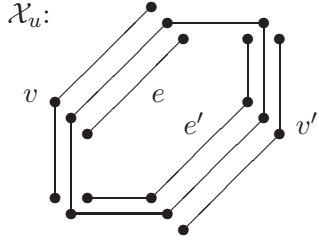
\begin{figure}
\setlength{\unitlength}{1.2pt}
\begin{picture}(200,80)(-150,52)
\put(-140,120){$\cX_u$:}
\put(-90,120){\circle*{3}} \put(-60,120){\circle*{3}} 
\put(-120,90){\circle*{3}} \put(-60,90){\circle*{3}} 
\put(-120,60){\circle*{3}} \put(-90,60){\circle*{3}} 
\put(-90,120){\line(1,0){30}} \put(-90,120){\line(-1,-1){30}} 
\put(-60,90){\line(0,1){30}} \put(-60,90){\line(-1,-1){30}} 
\put(-120,60){\line(0,1){30}} \put(-120,60){\line(1,0){30}} 
\put(-95,125){\circle*{3}} 
\put(-135,95){$v$}  
\put(-125,95){\circle*{3}} \put(-125,65){\circle*{3}} 
\put(-125,95){\line(1,1){30}} \put(-125,95){\line(0,-1){30}} 
\put(-55,85){\circle*{3}} \put(-55,115){\circle*{3}} 
\put(-85,55){\circle*{3}} \put(-50,85){$v'$} 
\put(-55,85){\line(-1,-1){30}} \put(-55,85){\line(0,1){30}} 
\put(-85,115){\circle*{3}}
\put(-115,85){\circle*{3}}
\put(-95,95){$e$}
\put(-115,85){\line(1,1){30}}
\put(-65,95){\circle*{3}} \put(-65,115){\circle*{3}}
\put(-95,65){\circle*{3}} \put(-115,65){\circle*{3}}
\put(-85,85){$e'$}
\put(-65,95){\line(-1,-1){30}} \put(-65,95){\line(0,1){20}}
\put(-95,65){\line(-1,0){20}}

\end{picture}
\caption{Link of vertex $u$, residues of $v,e$, and a 1-ball disjoint from each}
\label{fig:Z2_vtxlink}
\end{figure}

If $v,v'$ are vertices of that hexagon,
then each of $\Res(v,\cX_u)$ and $B_1(v',\cX_u)$ consists of the three vertices 
and two edges on a path on the perimeter of the hexagon.
For edges $e,e'$, $\Res(e,\cX_u)$ consists of $e$ and the two vertices
on $e$, whereas $B_1(e',\cX_u)$ consists of the four vertices and three edges
on a path on the perimeter of the hexagon, and contains the 1-balls of both of
the vertices on $e'$. 
It follows from the conditions on a directed geodesic 
that, if $\sigma_{i+1}$ is a vertex, then either
\begin{mylist}
\item[(i)] $\sigma_i$ and $\sigma_{i+2}$ are both vertices, opposite each 
other within the hexagon, or 
\item[(ii)]  $\sigma_i$ and $\sigma_{i+2}$ are both edges, opposite each 
other within the hexagon, or 
\item[(iii)] $\sigma_i$ is an edge and $\sigma_{i+2}$ is one of the two
vertices on the edge opposite $\sigma_i$ within the hexagon.
\end{mylist}

\begin{figure}
\setlength{\unitlength}{1.2pt}
\begin{picture}(200,150)(-50,-10)
\put(0,0){\circle*{3}} \put(30,0){\circle*{3}} \put(60,0){\circle*{3}} 
\put(90,0){\circle*{3}} \put(120,0){\circle*{3}} 
\put(150,0){\circle*{3}} \put(180,0){\circle*{3}} 
\put(0,30){\circle*{3}} \put(30,30){\circle*{3}} \put(60,30){\circle*{3}} 
\put(90,30){\circle*{3}} \put(120,30){\circle*{3}}  
\put(150,30){\circle*{3}} \put(180,30){\circle*{3}} 
\put(0,60){\circle*{3}} \put(30,60){\circle*{3}} \put(60,60){\circle*{3}} 
\put(79,63){$v_0$} \put(90,60){\circle*{1}} 
\put(120,60){\circle*{3}} \put(150,60){\circle*{3}} \put(180,60){\circle*{3}} 
\put(0,90){\circle*{3}} \put(30,90){\circle*{3}} \put(60,90){\circle*{3}} 
\put(90,90){\circle*{3}} \put(120,90){\circle*{3}} 
\put(150,90){\circle*{3}} \put(180,90){\circle*{3}} 
\put(0,120){\circle*{3}} \put(30,120){\circle*{3}} \put(60,120){\circle*{3}} 
\put(90,120){\circle*{3}} \put(120,120){\circle*{3}} 
\put(150,120){\circle*{3}} \put(180,120){\circle*{3}} 
\put(-10,0){\line(1,0){200}} \put(-10,30){\line(1,0){200}} \put(-10,60){\line(1,0){200}}
\put(-10,90){\line(1,0){200}} \put(-10,120){\line(1,0){200}} 
\put(0,-10){\line(0,1){140}} \put(30,-10){\line(0,1){140}} \put(60,-10){\line(0,1){140}} 
\put(90,-10){\line(0,1){140}} \put(120,-10){\line(0,1){140}} 
\put(150,-10){\line(0,1){140}} \put(180,-10){\line(0,1){140}} 
\put(-10,-10){\line(1,1){140}}
\put(20,-10){\line(1,1){140}}
\put(50,-10){\line(1,1){140}}
\put(80,-10){\line(1,1){110}}
\put(110,-10){\line(1,1){80}} \put(140,-10){\line(1,1){50}}
\put(170,-10){\line(1,1){20}}
\put(-10,20){\line(1,1){110}} \put(-10,50){\line(1,1){80}} \put(-10,80){\line(1,1){50}}
\put(-10,110){\line(1,1){20}} 
{\color{blue}
\thicklines
\put(95,62){\circle*{1}} \put(100,62){\circle*{1}} \put(105,62){\circle*{1}} 
\put(110,62){\circle*{1}} \put(115,62){\circle*{1}}  
\put(125,62){\circle*{1}} \put(130,62){\circle*{1}} \put(135,62){\circle*{1}} 
\put(140,62){\circle*{1}} \put(145,62){\circle*{1}}  
\put(155,62){\circle*{1}} \put(160,62){\circle*{1}} \put(165,62){\circle*{1}} 
\put(170,62){\circle*{1}} \put(175,62){\circle*{1}}  
\put(120,60){\circle*{5}} \put(150,60){\circle*{5}} \put(180,60){\circle*{5}} 
\put(200,60){\circle{12}}\put(197,59){${}_1$}
\put(85,53){\circle*{1}} \put(80,48){\circle*{1}} \put(75,43){\circle*{1}} 
\put(70,38){\circle*{1}} \put(65,33){\circle*{1}} \put(60,28){\circle*{1}} 
\put(55,23){\circle*{1}} \put(50,18){\circle*{1}}  \put(45,13){\circle*{1}} 
\put(40,8){\circle*{1}} \put(35,3){\circle*{1}}  \put(30,0){\circle*{1}} 
\put(60,30){\circle*{5}} \put(30,0){\circle*{5}}
\put(25,-10){\circle{12}}\put(23,-11){${}_1$}
\put(93,66){\circle*{1}} \put(96,72){\circle*{1}} \put(99,78){\circle*{1}} 
\put(102,84){\circle*{1}} \put(105,90){\circle*{1}}  
\put(108,96){\circle*{1}} \put(111,102){\circle*{1}} \put(114,108){\circle*{1}} 
\put(117,114){\circle*{1}} \put(120,120){\circle*{1}}  
\put(90,90){\line(1,0){30}} \put(120,120){\circle*{5}} 
\put(120,135){\circle{12}}\put(117,134){${}_2$}
\put(94,56){\circle*{1}} \put(98,52){\circle*{1}} \put(102,48){\circle*{1}} 
\put(106,44){\circle*{1}} \put(110,40){\circle*{1}} \put(114,36){\circle*{1}} 
\put(118,32){\circle*{1}} \put(122,28){\circle*{1}}  \put(126,24){\circle*{1}} 
\put(130,20){\circle*{1}}  \put(134,16){\circle*{1}} \put(138,12){\circle*{1}}  
\put(142,8){\circle*{1}} \put(146,4){\circle*{1}} 
\put(155,2){\circle*{1}} \put(160,2){\circle*{1}} 
\put(165,2){\circle*{1}} \put(170,2){\circle*{1}} \put(175,2){\circle*{1}} 
 \put(90,30){\line(1,1){30}} \put(120,30){\circle*{5}}  
\put(120,0){\line(1,1){30}} \put(150,0){\circle*{5}} \put(180,0){\circle*{5}}
\put(200,0){\circle{12}}\put(197,-1){${}_3$}
\put(84,57){\circle*{1}} \put(78,54){\circle*{1}} \put(72,51){\circle*{1}} 
\put(66,48){\circle*{1}} \put(60,45){\circle*{1}} \put(54,42){\circle*{1}} 
\put(48,39){\circle*{1}} \put(42,36){\circle*{1}}  \put(36,33){\circle*{1}} 
\put(60,30){\line(0,1){30}} \put(30,30){\circle*{5}}
\put(25,32){\circle*{1}}  \put(20,32){\circle*{1}} \put(15,32){\circle*{1}}  
\put(10,32){\circle*{1}} \put(5,32){\circle*{1}} 
\put(0,30){\circle*{5}}
\put(-5,20){\circle{12}}\put(-7,19){${}_3$}
\put(28,25){\circle*{1}}  \put(23,20){\circle*{1}} \put(18,15){\circle*{1}}  
\put(13,10){\circle*{1}} \put(8,5){\circle*{1}} 
\put(0,0){\circle*{5}}
\put(-5,-10){\circle{12}}\put(-7,-11){${}_3$}
}
\end{picture}
\caption{Directed geodesics in complex of $\Z^2$}

\end{figure}
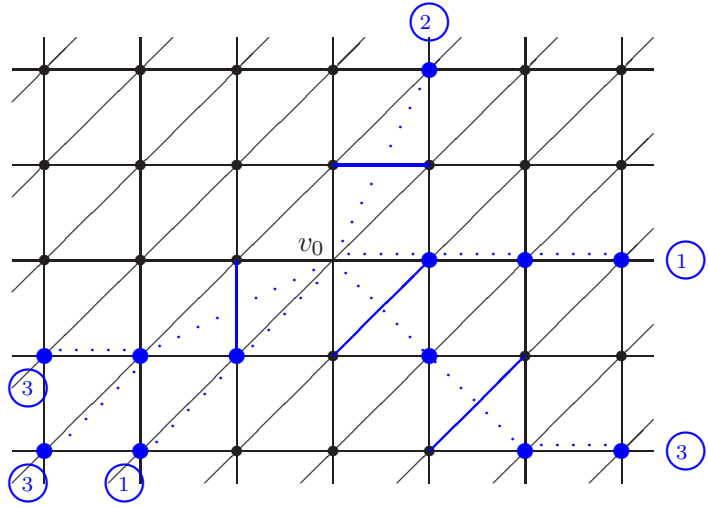

Hence we see that
the directed geodesics of the form $\gamma(v_0,gv_0)$
are of three different types, which can be described in terms of the diagram:
\begin{mylist}
\item[(1)] sequences of consecutive vertices $v_0,v_1,v_2,\ldots,v_n$ on one of the six rays through $v_0$, 
\item[(2)] sequences $v_0,\sigma_1,v_1,\sigma_2,v_2,\ldots,\sigma_k,v_k$ of 
alternately vertices and edges, where the vertices are all on a line out 
of $v_0$, the edges are bisected by that line, and the vertices and
edges are listed in the order in which they meet the line,
\item[(3)] sequences $v_0,\sigma_1,v_1,\ldots, \sigma_k,v_k,v_{k+1},\ldots,v_{k+m}$,
for which the sequence $v_0,\sigma_1,v_1,\ldots,\sigma_k,v_k$ is a sequence of type 2, and $v_k,v_{k+1}\ldots,v_{k+m}$ is the sequence of consecutive vertices on a ray that starts at $v_k$, one of the two rays through $v_k$ that are 
adjacent to the continuation of the line from $v_0$ to $v_k$.
\end{mylist}
Considering these three types,we can describe all the words in the language $\ctL$.
Corresponding to directed geodesics of type 1, we have all 
words $\alpha^n$ for $\alpha \in \cA$.
Corresponding to directed geodesics of type 2, we have all words 
of one of the forms
\[(ad)^n,\,(bd)^n,\,(\ainv b)^n,\,(\dinv\ainv)^n,\,(\dinv\binv)^n,\,(\binv a)^n.\] 
Corresponding to directed geodesics of type 3, we have all words of one of the forms
\begin{eqnarray*}
&&(ad)^ka^m,\,(ad)^kd^m,\,(bd)^kb^m,\,(bd)^kd^m,\,(\ainv b)^k\ainv^m,\,(\ainv b)^kb^m,\\
&&(\dinv\ainv)^k\ainv^m,\,(\dinv\ainv)^k\dinv^m,\,(\dinv\binv)^k\binv^m,\,
(\dinv\binv)^k\dinv^m,\,(\binv a)^ka^m,\,(\binv a)^k\binv^m.
\end{eqnarray*} 

\subsection{The right-angled Artin group $\Z^2 * \Z$}
\label{sec:Z2freeZ}
Let $G$ be the Artin group $\langle a, b, c \mid ab=ba\rangle$, isomorphic to
$\Z^2 * \Z$, and let $\cA =\{ 1,a,b,c,d,\ainv,\binv,\cinv,\dinv \}$, 
where $d$ represents the product $ab$, and $\ainv,\binv,\cinv,\dinv$ represent the inverses of $a,b,c,d$.
Let $G_{ab}$ be the subgroup $\langle a,b\rangle$, 
let $\cX_{ab}$ be the complex for $G_{ab}$ described in Section 
~\ref{sec:dihartZ2}, and let $\ctL_{ab}$ be the associated language
over $\{1,a,b,d,\ainv,\binv,\dinv\}$.
The systolic complex $\cX$ for $G$ that is built according to the construction of \cite{HuangOsajda} 
can be considered as a
`tree' of copies of $\cX_{ab}$; at each vertex of any given copy of $\cX_{ab}$
are attached edges labelled by $c$ and its inverse $\cinv$, whose targets are 
the basepoints (identity vertices) of further copies of $\cX_{ab}$, but at most 
one edge joins distinct copies of $\cX_{ab}$.  
Directed geodesics between two vertices of $\cX$ are formed as
concatenations of directed geodesics between vertices in copies of
$\cX_{ab}$ and paths along edges labelled $c$ or $\cinv$.
The biautomatic language $\ctL$ consists of all words of the form
$w_1\cdots w_k$, with $k \geq 1$, where for each $j$,
either $w_j \in \ctL_{ab}$ or $w_j=c^{i_j}$, for $i_j \neq 0$, 
alternately, and for $j>1$, $w_j$ is non-empty.

\subsection{The right-angled Artin group $\F_2 \times \Z$}
\label{sec:F2timesZ}
Let $G$ be the Artin group $\langle a, b, c \mid ab=ba,ac=ca\rangle$,
isomorphic to $\F_2 \times \Z$, and let
$\cA =\{ 1,a,b,c,d,e,\ainv,\binv,\cinv,\dinv,\einv \}$, 
where $d$, $e$ represent the products $ab$ and $ac$, and $\ainv,\binv,\cinv,\dinv,\einv$ represent the inverses of $a,b,c,d, e$.
Let $G_{ab}$, $G_{ac}$ be the subgroups $\langle a,b\rangle$ and 
 $\langle a,c\rangle$, let $\cX_{ab}$, $\cX_{ac}$ be the complexes for
$G_{ab}$, $G_{ac}$ described in Section \ref{sec:dihartZ2},
and let $\ctL_{ab}$, $\ctL_{ac}$ be the associated languages
over $\{1,a,b,d,\ainv,\binv,\dinv \}$ and $\{1,a,c,e,\ainv,\cinv,\einv \}$.
The systolic complex $\cX$ for $G$ that is built according to the construction of \cite{HuangOsajda} 
can be considered as a
`tree' of copies of $\cX_{ab}$ and $\cX_{ac}$; 
a copy of $\cX_{ac}$ is attached along each `ray' of $a$-edges within any given copy of $\cX_{ab}$, and similarly
a copy of $\cX_{ab}$ is attached along each `ray' of $a$-edges within any given copy of $\cX_{ac}$,
but copies of $\cX_{ab}$ and $\cX_{ac}$ can only intersect is a single $a$-ray, and
distinct copies of either $\cX_{ab}$ or $\cX_{ac}$ must have trivial 
intersection.

Since every subsequence of a directed geodesic must be a directed geodesic,
we can see that a directed geodesic between two vertices of $\cX$ must be a 
sequence $\sigma_1=\sigma_{i_0},\ldots,\sigma_{i_1},\ldots,\sigma_{i_j},\ldots,\sigma_{i_k}=\sigma_n$,
for which the subsequences $\gamma_j=\sigma_{i_{j-1}},\ldots,\sigma_{i_j}$
are directed geodesics within subcomplexes $\cX_{ab}$ and $\cX_{ac}$ 
alternately, and so in particular $\sigma_{i_1},\ldots,\sigma_{i_{k-1}}$
are all on intersection rays. The final subsequence might be along an 
intersection ray, but none of the others.
And, apart from $\sigma_0=\sigma_{i_0}$ and $\sigma_n=\sigma_{i_k}$, the simplices $\sigma_{i_j}$ need not necessarily be vertices.

To understand the precise form of such a sequence
we need to consider what must happen when a simplex $\sigma_i$ is
on a ray of $a$-edges, while $\sigma_{i-1}$, $\sigma_{i+1}$ are not
on the ray, but each is within (a distinct) one of the two subcomplexes containing that ray.

When $\sigma_i$ is a vertex $v=gv_0$, then its link is the union of two intersecting
hexagons, as shown in Figure \ref{fig:F2xZ_vtxlink}. In this example there are
five orbits of $G$ on edges, and we denote by 
$\xi=\{v_0,av_0\}$, $\eta=\{v_0,bv_0\}$, $\theta=\{v_0,cv_0\}$
$\delta=\{v_0,dv_0\}$, $\zeta=\{v_0,ev_0\}$,
the representatives of those
orbits that pass through the vertex $v_0$. 
The vertices $gav_0$ and $g\ainv v_0$ are on the intersection ray, and the
outer and inner hexagons are within 
the two subcomplexes that intersect on that ray. 

Suppose first that $\sigma_{i-1}$ is a vertex (not on the intersection ray).
Within the link, the residue of
any one of those eight vertices has the same shape, is a path of length 3
containing three vertices and two edges. This residue is disjoint from the
1-ball of precisely two vertices but of no edges  from the opposite subcomplex.

So, for example, if $\sigma_{i-1}=g\binv v_0$ (and $\sigma_i=gv_0$), then $\sigma_{i+1}=gcv_0$ or $g\einv v_0$; in the corresponding word, this situation gives us $\lambda_i=b$ and $\lambda_{i+1}=c$ or $\einv$, but not $\lambda_{i+1}=\cinv$ or $e$, so we see that words of the form $b^rc^s$ label directed geodesics
that are sequences of vertices, but those of the form $b^r\cinv^s$ do not.

Now suppose that $\sigma_{i-1}$ is an edge.
If $\sigma_{i-1}$ is one of the four edges in the orbit of $\xi$, 
then its residue within the link consists of just two vertices, and is disjoint 
from the 1-ball of two edges and of four vertices within the opposite subcomplex;
for example, if $\sigma_{i-1}= g\dinv \xi$ , then 
$\sigma_{i+1}$ could be any one of 
$gc\xi$, $g\einv \xi$, $gcv_0$, $gev_0$, $g\einv v_0$, or $g\cinv v_0$. 

But if $\sigma_{i-1}$ is one of the remaining eight edges, then its residue within 
the link is disjoint from the 1-ball of two edges and two vertices in the opposite
subcomplex; for example, if $\sigma_{i-1} = g\binv \delta$, then
$\sigma_{i+1}$ could be any one of $g\ainv \zeta$, $g\einv \theta$, $gcv_0$ or $g\einv v_0$.

When $\sigma_i$ is an edge of the intersection ray, then its link is a set of
four vertices, and $\sigma_{i-1}$ and $\sigma_{i+1}$ can be any vertices within
that link in distinct subcomplexes.

From the above analysis, we see from the description in
Section~\ref{sec:dihartZ2} that each of the subsequences $\gamma_j$ can be
either (1) a sequence of vertices, or (2) a sequence of edges and vertices
alternately, or (3) a sequence of type (2) followed by one of type (1).
$\gamma_{j+1},\ldots,\gamma_k$ must all be of type (1). It is possible that
$\gamma_j$ might end with an edge, which would then be the first simplex in
$\gamma_{j+1}$.
Otherwise there are some restrictions on which subwords
$\lambda_{i_j}\lambda_{{i_j}+1}$ we might see at points within words in $\ctL$
that correspond to the junction point of two concatenated directed geodesics.

\begin{figure}
\setlength{\unitlength}{2.5pt}
\begin{picture}(200,68)(-150,55)
\put(-125,120){$\cX_v$:}
\put(-90,120){\circle*{2}} \put(-60,120){\circle*{2}} 
\put(-90,123){$gbv_0$} \put(-60,123){$gdv_0$} 
\put(-120,90){\circle*{2}} \put(-60,90){\circle*{2}} 
\put(-130,90){$g\ainv v_0$} \put(-58,90){$gav_0$} 
\put(-120,60){\circle*{2}} \put(-90,60){\circle*{2}} 
\put(-130,60){$g\dinv v_0$} \put(-86,60){$g\binv v_0$} 
\put(-90,120){\line(1,0){30}} \put(-90,120){\line(-1,-1){30}} 
\put(-78,122){$gb\xi$} \put(-112,105){$g\ainv \delta$} 
\put(-60,90){\line(0,1){30}} \put(-60,90){\line(-1,-1){30}} 
\put(-59,105){$ga\eta$} \put(-75,70){$gB\delta$} 
\put(-120,60){\line(0,1){30}} \put(-120,60){\line(1,0){30}} 
\put(-128,75){$g\dinv\eta$} \put(-110,57){$g\dinv \xi$} 
\put(-90,110){\circle*{2}} \put(-70,110){\circle*{2}} 
\put(-90,107){$gcv_0$} \put(-78,107){$gev_0$} 
\put(-110,70){\circle*{2}} \put(-90,70){\circle*{2}} 
\put(-110,73){$g\einv v_0$} \put(-95,73){$g\cinv v_0$} 
\put(-90,110){\line(1,0){20}} \put(-90,110){\line(-3,-2){30}} 
\put(-82,112){$gc\xi$} \put(-109,95){$g\ainv \zeta$} 
\put(-60,90){\line(-1,2){10}} \put(-60,90){\line(-3,-2){30}} 
\put(-71,95){$ga\theta$} \put(-83,80){$g\cinv \zeta$} 
\put(-110,70){\line(-1,2){10}} \put(-110,70){\line(1,0){20}} 
\put(-114,80){$g\einv \theta$}  \put(-105,67){$g\einv \xi$}  
\end{picture}
\caption{Link of a vertex on an intersection ray}
\label{fig:F2xZ_vtxlink}
\end{figure}

\subsection{The dihedral Artin group $G(A_2)=\langle a,b \mid aba=bab \rangle $}
\label{sec:dihartA2}
The systolic complex constructed as in \cite{HuangOsajda} for the 2-generator
Artin group of type $A_2$ is already significantly more complicated than the
complex for $\Z^2$. 
Each of the precells in the Cayley graph is triangulated by the addition of one
vertex, six edges and six 2-cells. Each new vertex then has a link that is a
6-cycle, but the process also creates 4-cycles in the links of some of the
original vertices as was shown in Figure~\ref{fig:badlink}.

Those bad links are corrected by the addition of edges, and then the
attachment of 2-cells and 3-cells,
so that the complex continues to be flag. 
The process terminates with a
3-dimensional systolic complex, which has three orbits of vertices, ten
orbits of edges, twelve orbits of 2-cells and four orbits of 3-cells.

We have $\cB = \cA =\{1,a,b,c,d,e,\ainv,\binv,\cinv,\dinv,\einv\}$,
where $c$, $d$ and $e$ represent the group elements $ab$, $ba$ and $aba$,
respectively and, as before, $\ainv$ represents the inverse of $a$, etc.
The minimised deterministic finite state automaton with accepted language
$\ctL$ has $51$ states, and we are unable to provide a useful description
of $\ctL$. Instead, we shall just give some examples of words in the
language. In the following description, when we refer to the `precell
based at $g$' for an element $g \in G$, we mean the precell of which the
initial vertex is the real vertex labelled $g$.

We have $a^n, \ainv^n, b^n, \binv^n \in \ctL$ for all $n \ge 0$, and these
label corresponding directed geodesics through real vertices of the complex.
But the accepted word for the group element $ab$ is $1c$. A
directed geodesic with that label goes from $v_0$ to a 2-cell that
contains the interior vertices of the precells based at $1,a$ and $\binv$, and
from there to the real vertex labelled $ab$.
Similarly, the accepted word for $ba$ is $1d$. The word for $a^{-1}b^{-1}$
is $\dinv 1$, and labels a corresponding directed geodesic in the reverse
direction.  Similarly, $\cinv 1 \in \ctL$.

The accepted word for $aba$ is $1e$, with corresponding directed geodesic
starting from $v_0$ and passing through the interior vertex of the precell
based at $1$. In fact $(1e)^n\in \ctL$ for all $n \ge 0$.
The word for $(aba)^{-n}$ is $(\einv 1)^n$, and labels the corresponding directed
geodesic in the opposite direction.

The word in $\ctL$ for $ab^{-1}$ is $\cinv d$. It seems easiest to describe
a corresponding directed geodesic with this label as one that starts at the
real vertex labelled $c$, proceeds to the interior edge (1-cell) that
joins the interior vertex of the precell based at $1$ to the real vertex
labelled $e$, and from there to the real vertex labelled $d$. Similarly,
the word in $\ctL$ for $a^{-1}b$ labels a directed geodesic from the real vertex
labelled $a$ to the real vertex labelled $b$, passing through an interior edge
within the precell based at $1$.

The word in $\ctL$ for $a^2b$ is $1ac$, with directed geodesic going from
$1$ to the interior edge joining the real vertex labelled $a$ and
the interior vertex for the precell based at $1$, and from there to the interior
vertex of the precell based at $a$, and finally to the real vertex labelled
$a^2b$.  A directed geodesic labelled $1ae$ has the same first two edges but
the final edge leads to the real vertex labelled $(ab)^2=_G a^2ba =_G ae$.

These examples suggest that it might be easier to understand the
nature of the directed geodesics rather than the words that label them, which
is not surprising given that the labelling words depend on a choice of orbit
representatives of the simplices, whereas the directed geodesics themselves
have no such dependency.

\section{Properties of the Artin group structures}
\label{sec:properties}

From now on we suppose that $G$ is a systolic Artin group, and $\cX$
the associated systolic complex.
We will always assume that $v_0$ is the vertex of $\cX$ that is
equal to the basepoint of the embedded Cayley graph $\Cay$, and labelled by the
identity element of $G$.
Since all simplex stabilisers are trivial, each of the sets $\Lambda_\sigma$ is
a singleton set; we denote by $\lambda_\sigma$ its single element.

It is clear that the structure of $\ctL$ must depend on our choice of the set $V_0$ of orbit representatives, since this determines the labels. But we can make a choice of $V_0$
that imposes a sensible structure, as we see below.

Suppose first that we have a 2-generator Artin group.
Let $\Pi_0$ be the unique precell that has $v_0$ as its initial vertex.
We say that $V_0$ is {\em based on $\Pi_0$} if 
\begin{mylist}
\item[(1)] $V_0$ contains $v_0$ and the two real edges through $v_0$ on the
boundary of $\Pi_0$;
\item[(2)] $V_0$ contains all the interior vertices and edges of $\Pi_0$;
\item[(3)] any simplex in $V_0$ contains at least one vertex within $\Pi_0$.
\end{mylist} 

The following lemma is straightforward to prove, and is used in the proof of
the proposition that follows.
\begin{lemma}
Suppose that $v_0$ is the identity vertex, and $V_0$  is based on $\Pi_0$. Then
\begin{mylist}
\item[(i)] the label $\lambda_v$ of a real vertex of $\cX$ is the same as its
label as a vertex of $\Cay$;
\item[(ii)] the label $\lambda_v$ of an interior vertex of $\cX$ is the
same as the label of the initial vertex of the unique precell containing $v$;
that is, $\lambda_v=g$ where $v$ is an interior vertex of $g\Pi_0$.
\end{mylist}
\end{lemma}

\begin{proposition}
\label{prop:dih_cA_Garside}
Let $G$ be the 2-generator Artin group 
\[ \langle a, b \mid {}_m(a,b)= {}_m(b,a) \rangle. \]
If $V_0$ is based on $\Pi_0$, then each of the sets $\cA$ and $\cB$ consists
of (a set representing) the complete set of Garside generators of $G$,
together with their inverses and the identity element.
\end{proposition}
\begin{proof}

We represent the set of Garside generators of $G$ as 
the set of all alternating products of $a,b$ of length less than $m$, together
with the unique alternating product of length $m$ that begins with $a$, which 
we denote by $\Delta$ (see, e.g \cite{Charney2}). 
If $g$ is a Garside generator, then so are $g^{-1}\Delta$ and $\Delta g^{-1}$.

We start by observing a correspondence between the vertices other than $v_0$ on
the boundary of $\Pi_0$ and the Garside generators; for each
real vertex $v\neq v_0$ on the boundary of $\Pi_0$ the label $\lambda_v$ 
of $v$ (which maps $v_0$ to $v$)
is a  Garside generator, and for each Garside generator $g$, $gv_0$ is
on the boundary of $\Pi_0$.
The generator $\Delta$ is the label of the terminal vertex of $\Pi_0$.

Let $\ainv := a^{-1}$ and $\binv := b^{-1}$.
We see that $a,b,\ainv,\binv$ must all be in $\cA$, by considering the directed
godesics of length 1 from $v_0$ to each of
$av_0$, $bv_0$, $\ainv v_0$, $\binv v_0$, and by considering the related
polygonal path, we see that each of
$a,b,\ainv,\binv$ must also be in $\cB$.

We observe that a directed geodesic of length $m-1$ joins $v_0$ to $\Delta v_0$,
through the $m-2$ interior vertices of $\Pi_0$.
Since all interior vertices and edges of $\Pi_0$ are in $V_0$,
we see that the directed geodesic corresponds to
the word $1\cdots 1\Delta$ of length $m-1$ over $\cA$ in $\ctL$, and 
then that the polygonal path defined by that directed geodesic corresponds 
to the word $11 \cdots 11\Delta$ of length $2m-1$ over $\cB$ in $\cL'$. 
Similarly the word $\Delta^{-1}1\cdots 1$ of length $m-1$ over $\cA$ in $\ctL$
and the word $\Delta ^{-1}11\cdots 11$ of length $2m-1$ over $\cB$ in $\cL'$
label the directed geodesic and polygonal path from $v_0$ to $\Delta^{-1}v_0$. 
So we see that $1,\Delta,\Delta^{-1}$ are all in both $\cA$ and $\cB$.

Now let $v,w$ be vertices, one on each of the two half boundaries of $\Pi_0$, at distances $i,i-1$
from $v_0$, for some $1<i<m$. Then $\lambda_v,\lambda_w$ are Garside
generators, and in fact $\lambda_w$ is the maximal proper suffix of $\lambda_v$.
A path of length $2$ joins $v$ to $w$ in $\cX_1$, through an interior vertex 
$u$ of $\Pi_0$; both $w,u,v$ and $v,u,w$ are directed geodesics,
and their images under $\lambda_w^{-1}$ and $\lambda_v^{-1}$ are
directed geodesics from $v_0$ to $\lambda_w^{-1}v$ and $\lambda_v^{-1}w$
respectively, labelled by words $\lambda_w^{-1}\lambda_v$ and 
$\lambda_v^{-1}\lambda_w$ in $\ctL$.
We deduce that $\lambda_v$, $\lambda_v^{-1}$ (and also $\lambda_w$,
$\lambda_w^{-1}$) are in $\cA$, and similarly in $\cB$.

Conversely we need to verify that whenever $\sigma,\rho$ are consecutive
simplices in a
directed geodesic or polygonal path then $\lambda_{\sigma}^{-1}\lambda_{\rho}$ is
equal to either $1$, a Garside generator, or the inverse of such.
We have
$\sigma = \lambda_\sigma \bar{\sigma}$ and $\rho = \lambda_\rho \bar{\rho}$. Then 
it follows from condition (2) above on $V_0$ that
each of $\bar{\sigma}$, $\bar{\rho}$ must contain a vertex of $\Pi_0$.
It follows that $\sigma$ contains a vertex $v$ with $\lambda_v=\lambda_\sigma$, and $\rho$ a vertex $w$ with $\lambda_w=\lambda_\rho$.
Now, since $\sigma,\rho$ are consecutive vertices in either a directed
geodesic or a polygonal path, $v,w$ are vertices within a simplex of $\cX$
(within $\sigma*\rho$ in the first case, and one of $\sigma,\rho$ in the second).
If $v=w$, then $\lambda_\sigma=\lambda_v=\lambda_\rho$.
Otherwise, $\{v,w\}$ is an edge of $\cX_1$. If both $v$ and $w$ are real,
then they
are adjacent vertices of $\Cay$, and so $\lambda_\sigma^{-1}\lambda_\rho=\lambda_v^{-1}\lambda_w \in \{a,b,\ainv,\binv\}$.
If $v$ is real and $w$ is interior, then $v$ must be a vertex on the boundary
of the unique precell $\lambda_w\Pi_0$ that contains $w$, and then
$\lambda_v = \lambda_w g$, so $\lambda_\sigma^{-1}\lambda_\rho=g^{-1}$
where either $g=1$ or $g$ is a Garside generator.
If $v$ is interior and $w$ real, the same argument gives $\lambda_\sigma^{-1}\lambda_\rho=g$, with $g=1$ or $g$ a Garside generator.

Finally, if $v$ and $w$ are both interior, then the unique precells
$\lambda_v\Pi_0$ and $\lambda_w\Pi_0$ that contain them are either equal,
in which case $\lambda_v=\lambda_w$, or they have boundaries that
intersect in a path of length at least 2 that connects the initial vertex of one of
the two precells (labelled by $\lambda_v$ or $\lambda_w$) to the terminal vertex 
of the other (labelled by $\lambda_w\Delta$ or $\lambda_v\Delta$). The element
$g$ labelling such a path must be a Garside generator, and we
have either $\lambda_w\Delta=\lambda_vg$ or $\lambda_v\Delta=\lambda_wg$,
from which we see that $\lambda_\sigma^{-1}\lambda_\rho=
\lambda_v^{-1}\lambda_w$ is equal either to $\Delta g^{-1}$,
which is also a Garside generator, or to its inverse.
\end{proof}

We easily deduce  the following corollary from
Proposition~\ref{prop:dih_cA_Garside}.
\begin{corollary}
\label{cor:cA_Garside}
Let $G$ be an Artin group of
almost large type, and $\cX$ the systolic complex for $G$ constructed as in
\cite{HuangOsajda}. Select a vertex $v_0$ of $\cX$,
and suppose that, for each $i,j$,
$\Pi_0^{ij}$ is a precell within $\cX_{ij}$ with initial vertex $v_0$.
Suppose that $V_0$ is chosen such that orbit representatives for each parabolic subgroup
$G_{ij}$ of $G$ are based on $\Pi_0^{ij}$.
Then each of the sets $\cA$ and $\cB$ consists
of (a set representing) the union of complete sets of Garside generators for the subgroups $G_{ij}$, together with their inverses
and the identity element.
\end{corollary}

We now consider properties of the language $\ctL$ over $\cA$.

We can find upper and lower bounds on the lengths of words in $\ctL$. The upper 
bounds are general for a systolic group with trivial vertex stablisers, but the 
lower bound is specific to systolic Artin groups.
\begin{proposition}
\label{prop:wordlength}
Let $G=\langle X \rangle$ be systolic, and $g \in G$,
and let $\cL$ be the language of a biautomatic structure over $\cA$.
\begin{mylist}
\item[(i)] If $w' \in \cL$ represents $g$, then  $|w'|_\cA \leq |g|_X + 1$.
\item[(ii)] If $G$  has trivial vertex stabilisers and $w \in \ctL$ represents
$g$, then $|w|_\cA \leq |g|_X$.
\item[(iii)] If $G$ is a non-free Artin group of almost large type and $w \in \ctL$
represents $g$, then
$|g|_X \leq \max\{2,M-2\}|w|_\cA$,
where $M$ is the maximum of those $m_{ij}$ that are finite.
\end{mylist}
\end{proposition}
\begin{proof}
For (i) and (ii), we note that if
$\sigma_0,\ldots,\sigma_n$ is the directed geodesic that corresponds to $w' \in \cL$, then $|w'|=n+1$.
But now if
$\gamma= v_0,\ldots,v_n$ is a corresponding allowable geodesic (so that 
$v_i \in \sigma_i$), then
it follows from \cite[Fact 11.1]{JS} that $\gamma$ is a geodesic in the 1-skeleton of $\cX$. So, since the Cayley graph $\Cay$ embeds naturally within $\cX_1$,
we have $n \leq |g|_X$, and hence (i) follows.
Now if $w \in \ctL$ representing $g$, then the word $1w$ represents $g$ in
$\cL$, and so (ii) is an immediate consequence of (i).
Finally, (iii) follows by Lemma~\ref{lem:pathlength}
\end{proof}

\begin{proposition}
\label{prop:extend_dirgeo}
If $\sigma_0,\ldots,\sigma_n$ is a directed geodesic,
then there is a directed geodesic $\sigma_0,\ldots,\sigma_{n-1},v_n$,
for which $v_n$ is a vertex. Furthermore, if
$v_n$ is an interior vertex, then there is a directed geodesic
\[\sigma_0,\ldots,\sigma_{n-1},v_n,v_{n+1},\ldots,v_{n+k},\] with
$v_{n+k}$ real and $k\leq M-2$, where $M$ is the maximum of those $m_{ij}$ that are finite.
\end{proposition}
\begin{proof}
The existence of a vertex $v_n$ such that 
$\sigma_0,\ldots,\sigma_{n-1},v_n$ is a directed geodesic
follows directly from Lemma~\ref{lem:subpath_dir_geo}\,(2);
then $\sigma_{n-1}$ is within the link $\cX_{v_n}$ of $v_n$.
If $v_n$ is not real, then it is an interior vertex within a precell $\Pi$
of some $\cX_{ij}$ subcomplex.
By Lemma~\ref{lem:Xij_newedges} the link of $v_n$ is also within that
subcomplex: it intersects
the precell in a hexagon, as shown in Figure~\ref{fig:lnk_in_precell}.
Let $u,w,v_n, v',v''$ be the vertices marked in that figure.

Let $\pi$ be the intersection of
$\Res(\sigma_{n-1},\cX_{v_n})$ with that hexagon.
We claim that $\pi$ is either 
a closed path of two consecutive edges and their vertices, or
is a single closed edge.
That is clear if $\sigma_{n-1}$ is within $\Pi$, in which case $\sigma_{n-1}$ is
either a vertex of edge of the hexagon.
Otherwise, $\sigma_{n-1}$ contains a vertex $u'$ outside $\Pi$ that is joined 
by an edge to $v_n$. Then, by Lemma~\ref{lem:Xij_newedges},
$u'$  must be an interior vertex in a precell $\Pi'$
that intersects $\Pi$ along a path on one of its half boundaries.
Then every vertex of $\pi$ is joined to $u'$ by an edge, and so in this case our
claim follows by Lemma~\ref{lem:Xij_newedges}. 

Now, except when $\pi$ consists of the two closed edges through either $v'$
or $v''$, at least one of the four real vertices of the hexagon is distance
at least two from $\pi$, and so its 1-ball is disjoint from $\pi$. In that
case, we can choose $v_{n+1}$ to be such a vertex, and
$\sigma_0,\ldots,\sigma_{n-1},v_n, v_{n+1}$ is a directed geodesic
with $v_{n+1}$ real.

Otherwise, without loss of generality, $\pi$ consists of the two closed edges
through $v'$. Then the 1-ball about $v''$ is disjoint from $\pi$,
and we can choose $v_{n+1}$ to be $v''$.  Now
$\sigma_0,\ldots,\sigma_{n-1},v_n, v_{n+1}$ is a directed geodesic,
but $v_{n+1}$ is interior. 
But now, by the same argument as above,
we can extend the directed geodesic with vertices $v_{n+2},\ldots,v_{n+k}=w$ 
as successive vertices on the path within the triangulated precell from $v''$ to $w$.
\end{proof}

\begin{figure}
\setlength{\unitlength}{0.8pt}
\begin{picture}(300,60)(-60,-30)
\put(0,0){\circle*{3}} \put(-10,2){$u$}
\put(240,0){\circle*{3}} \put(242,2){$w$}

\put(40,0){\circle*{3}} 
\put(45,0){\circle*{1}} \put(50,0){\circle*{1}} \put(55,0){\circle*{1}} 
\put(60,0){\circle*{1}} \put(65,0){\circle*{1}} \put(70,0){\circle*{1}}
\put(75,0){\circle*{1}}
\put(80,0){\circle*{3}} \put(72,2){$v'$}
\put(120,0){\circle*{3}} \put(127,4){$v_n$} 
\put(160,0){\circle*{3}} \put(162,2){$v''$}
\put(165,0){\circle*{1}} \put(170,0){\circle*{1}} \put(175,0){\circle*{1}} 
\put(180,0){\circle*{1}} \put(185,0){\circle*{1}} \put(190,0){\circle*{1}} 
\put(195,0){\circle*{1}}
\put(200,0){\circle*{3}} 

\put(0,0){\line(1,0){40}} 
\put(80,0){\line(1,0){40}} \put(120,0){\line(1,0){40}} 
\put(200,0){\line(1,0){40}}

\put(20,30){\circle*{3}} \put(60,30){\circle*{3}}
\put(65,30){\circle*{1}} \put(70,30){\circle*{1}}  \put(75,30){\circle*{1}} 
\put(80,30){\circle*{1}} \put(85,30){\circle*{1}}  \put(90,30){\circle*{1}}  
\put(95,30){\circle*{1}}   
\put(100,30){\circle*{3}}  \put(140,30){\circle*{3}} 
\put(145,30){\circle*{1}}
\put(150,30){\circle*{1}} \put(155,30){\circle*{1}} \put(160,30){\circle*{1}} 
\put(165,30){\circle*{1}} \put(170,30){\circle*{1}} \put(175,30){\circle*{1}} 
\put(180,30){\circle*{3}} \put(220,30){\circle*{3}}

\put(0,0){\line(2,3){20}} 
\put(20,30){\line(1,0){40}} 
\put(100,30){\line(1,0){40}} \put(180,30){\line(1,0){40}}
\put(240,0){\line(-2,3){20}} 

\put(40,0){\line(-2,3){20}} \put(40,0){\line(2,3){20}}
\put(80,0){\line(2,3){20}}
\put(120,0){\line(-2,3){20}} \put(120,0){\line(2,3){20}}
\put(160,0){\line(-2,3){20}}
\put(200,0){\line(-2,3){20}} \put(200,0){\line(2,3){20}}

\put(20,-30){\circle*{3}} \put(60,-30){\circle*{3}} 
\put(65,-30){\circle*{1}} \put(70,-30){\circle*{1}} \put(75,-30){\circle*{1}}
\put(80,-30){\circle*{1}} \put(85,-30){\circle*{1}} \put(90,-30){\circle*{1}}
\put(95,-30){\circle*{1}}
\put(100,-30){\circle*{3}} \put(140,-30){\circle*{3}} \put(145,-30){\circle*{1}} 
\put(150,-30){\circle*{1}} \put(155,-30){\circle*{1}} \put(160,-30){\circle*{1}} 
\put(165,-30){\circle*{1}} \put(170,-30){\circle*{1}} \put(175,-30){\circle*{1}}   
\put(180,-30){\circle*{3}} \put(220,-30){\circle*{3}}

\put(0,0){\line(2,-3){20}} 
\put(20,-30){\line(1,0){40}} 
\put(100,-30){\line(1,0){40}} 
\put(180,-30){\line(1,0){40}}
\put(240,0){\line(-2,-3){20}} 

\put(40,0){\line(-2,-3){20}} \put(40,0){\line(2,-3){20}}
 \put(80,0){\line(2,-3){20}}
\put(120,0){\line(-2,-3){20}} \put(120,0){\line(2,-3){20}}
\put(160,0){\line(-2,-3){20}} 
\put(200,0){\line(-2,-3){20}} \put(200,0){\line(2,-3){20}}
\end{picture}
\caption{Extending a directed geodesic past $v_n$}
\label{fig:lnk_in_precell}
\end{figure}
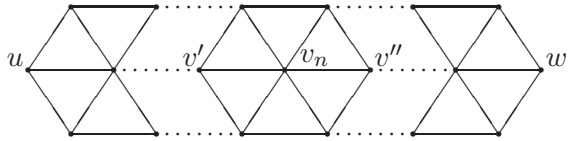

We note that, since the reverse of a directed geodesic is not in general a directed geodesic, we should not expect $\ctL$ to be symmetric. 
And although every prefix of a directed geodesic is also a directed geodesic,
its final simplex is not in general a vertex, and so we do not expect (or observe, in our examples) prefix closure. However, we have the following result.

\begin{corollary}
\label{cor:prefix}
Suppose that $w \in \ctL$, and that $w_1$ is an prefix of $w$,
and let $w_2$ be the maximal prefix of $w_1$ (formed by deleting its last letter).
Then there exists a word $w' \in \ctL$,
that has $w_2$ as a prefix, with 
$|w'|-|w_2|\leq \max_{i,j}(\{m_{ij} -1 : m_{ij} \ne \infty \})$.
\end{corollary} 
\begin{proof} 
We let $\sigma_0,\ldots,\sigma_n$ be the directed geodesic corresponding to
$w'$, and choose $w' \in \ctL$ to correspond to the directed
geodesic $\sigma_0,\ldots,\sigma_{n-1},v_n,\ldots,v_{n+k}$ that is constructed 
in Proposition~\ref{prop:extend_dirgeo}.
\end{proof}

\section{Computer code}
\label{sec:computer_code}
We have written computer code that constructs the biautomatic structures
described in \cite{HuangOsajda} for certain Artin groups $G$ of almost large
type. Currently our programs can handle such groups provided that all of the
entries $m_{ij},\,i \neq j$ in the associated Coxeter matrix lie in
$\{2,3,4,\infty\}$.
This code is written in $\GAP$ \cite{GAP} and makes use of the $\KBMAG$
\cite{KBMAG} package for computing shortlex automatic structures and
performing operations on finite state automata.

The biautomatic structures that we construct are over the generating set
$\cA$ that was described in Section \ref{sec:gensets} above for general
systolic groups.
With this code we can explictly construct the word acceptor and
right and left multiplier automata and verify computationally that they do
indeed define a biautomatic stucture. We can also reduce words over the
standard generating set $X = \{a_1, \ldots,a_n\}$ of $G$ to their equivalent
normal form words over $\cA$.

As we saw in Corollary \ref{cor:cA_Garside}, with appropriate choice of
the set $V_0$ of orbit represenatives,
the set $\cA$ is equal to (symbols representing) the set
$\{1\} \cup Y \cup Y^{-1}$, where $Y$ is the union
of $X$ and, for each $i,j$ with $1 \le i < j \le n$, the sets
corresponding to $\{a_ia_j\}$, $\{a_ia_j, a_ja_i, a_ia_ja_i\}$, and
$\{a_ia_j, a_ja_i, a_ia_ja_i, a_ja_ia_j, a_ia_ja_ia_j \}$,
when $m_{ij} =2,3$ and $4$, respectively.
In fact we only construct the left and right multiplier automata for
the generators in $\{ 1 \} \cup X \cup X^{-1}$, but this is sufficient
to verify the correctness of the structures, and the multipliers for the
additional generators could be constructed as composite automata
(defined in \cite[Section 5]{EHR}) if required.

The code works as follows. The authors proved in \cite{SLartin} and
\cite{SLartin2} that $G$ is shortlex automatic over $X$, and we start
by using $\KBMAG$ to construct this shortlex automatic structure.
Then we enumerate words in the accepted language up to some chosen length
$l$ (it turns out that $l=7$ is sufficient in general, or $l=4$ if there
there are no $i,j$ with $m_{ij}=4$).
So we have effectively constructed the ball of radius $l$ about the base point
in the Cayley graph $\Cay$ of $G$ over $X$. Using this, we construct the simplices in
part of the systolic complex $\cX$, where $v_0$ is the base point of $\Cay$.

Now, provided that $l$ is large enough, this part of the complex
contains enough information to enable us to construct the automaton,
as described in Section~\ref{sec:construct2}; since vertex stabilisers are trivial,
we construct the automaton $\ctM$ that accepts $\ctL$.
To do this, we first have to choose orbit representatives $V_0$ for the action 
of $G$ on $\cX$, which must satisfy the hypotheses of 
Corollary \ref{cor:cA_Garside}.
In general, the representatives are chosen to be close in $\cX$ to $v_0$.

The automaton $\ctM$ is in general non-deterministic, but we
can use the automata manipulation functions in $\KBMAG$ to construct
a deterministic automaton with a minimal number of states that accepts the same
language.

We next describe how we construct the multiplier automata.
In general, let $W$ be a finite subset of a group $G$, and let $u$ and $v$ be
words over some  generating set $X$ of $G$. For
$0 \le i \le \ell(u)$, let $u(i)$ denote the prefix of $u$ of length $i$,
and put $u(i)=u$ for $i \ge \ell(u)$.  Then, for $a \in G$, we write
$u \sim_{W,a} v$ if $u(i)^{-1}av(i) \in W$ for all $i \ge 0$.

For the language $L$ of an automatic structure of $G$ over $X$, there is a
finite set $W_R \subset G$ such that, for all $a \in X \cup X^{-1}$,
the language $\{(u,v) : u,v \in L,\,ua =_G v,\, u \sim_{W_R,1} v \}$ is regular
and has both of its components projecting onto $L$.
These are the languages of the right multipliers in the automatic structure.
An automaton accepting this language is described in \cite[Section 6.3]{ECHLPT}.

If in addition $L$ is the language of a biautomatic structure for $G$,
then there is a finite set $W_L \subset G$ such that, for all
$a \in X \cup X^{-1}$,
the language $\{(u,v) : u,v \in L,\,u =_G av,\, u \sim_{W_L,a} v \}$ is
regular and has both of its components projecting onto $L$.
These are the languages of the left multipliers in the biautomatic structure.

If we are given the word acceptor of an automatic structure and a candidate
for the set $W_R$, then we can construct the associated right multiplier
automata and then use the {\em axiom checking process} described in
\cite[Section 5.1]{ECHLPT} to verify that this really is an automatic structure
and hence that $W_R$ has the required properties. If the axiom checking
process fails, then we can attempt to find specific instances of failure
and thereby add further words to $W_R$, after which we try again.
This is essentially the same method that is described in more detail in
\cite{EHR} for constructing shortlex automatic structures.
As mentioned earlier, in our code we just construct the automata for
generators in $\{1\} \cup X \cup X^{-1}$, which is sufficient for the
axiom checking.

Furthermore, if we are given a candidate for $W_L$, then we can construct the
associated left multipliers. We can use these to check the condition
\[ \forall u \in L\, \exists v \in L\, (u=_G av,\,u \sim_{W_L,a} v) \]
for each $a \in X$. If the condition holds, then we have verified that
the structure is indeed biautomatic. If not, then we can find
specific words $u \in L$ for which the condition fails, compute words
$v \in L$ with $u=_G av$, and then add new elements to $W_L$ to make
the condition $u \sim_{W_L,a} v$ hold for these words.

Note that to do all of this in our specific situation for Artin groups (in
which the accepted language of the word acceptor is $\ctL$), we need to be able
to test words over $\cA$ for equality in $G$ (i.e. to solve the word problem in
$G$), but we can use the shortlex automatic structure to do that. We also need
to be able to reduce words to normal form; that is, given a word $u$ find a
word $v$ with $v \in \ctL$ such that $u =_G v$. A method for doing this that
uses the set $W_R$ is described in the proof of Theorem 4.1 of \cite{EHR}.

In practice it turned out that $W_R = \cA$ was adequate for the construction
of the right multipliers for the generators in  $\{1\} \cup X \cup X^{-1}$,
and also for word reduction.
But the set $W_L$ needs to be significantly larger, and we were able to
construct that using the error correcting process described above.


The principal reason that we have not attempted to handle Artin groups with
$m_{ij} > 4$ is that the complex $\cX$ becomes increasingly more difficult
to describe explicitly with increasing $m_{ij}$. We found that, for
$m_{ij}=2$, 3 and 4, the highest dimension of a simplex in the complex
is $m_{ij}$ and it seems plausible that the same is true for higher $m_{ij}$.


\begin{thebibliography}{99}
\bibitem{BradyMcCammond} T. Brady and J.P. McCammond, Three-generator Artin
groups of large type are biautomatic, J. Pure Appl. Alg. 151 (2000) 1--9.
\bibitem{Charney} R. Charney, Artin groups of finite type are biautomatic, Math. Annalen 292 (1992) 671--683.
\bibitem{Charney2} R. Charney, Geodesic automation and growth functions for 
Artin groups of finite type.
Math. Ann. 301 (1995), no. 2, 307--324
\bibitem{ECHLPT} David B.\,A. Epstein, J.\,W. Cannon, D.\,F. Holt, S. Levy,
M.\,S. Patterson and W. Thurston,{\em Word processing in groups}, Jones
and Bartlett, 1992.
\bibitem{EHR} D.\,B.\,A. Epstein, D.\,F. Holt and S.\,E. Rees,
The use of Knuth-Bendix methods to solve the word problem in automatic
groups, J. of Symbolic Computation, 12 (1991), 397--414.
\bibitem{GAP}
The GAP~Group.
{\em {GAP --- Groups, Algorithms, and Programming, Version 4.8.10}},
  2018.
\verb+(http://www.gap-system.org)+.
\bibitem{GerstenShort1} S.\,M. Gersten and H.\,B. Short,
Small cancellation theory and automatic groups. Invent. Math. 102 (1990), no. 2, 305--334.
\bibitem{GerstenShort2} S.\,M. Gersten and H.\,B. Short,
Small cancellation theory and automatic groups: Part II. Invent. Math. 105 (1991), no. 3, 641--662. 
\bibitem{HermillerMeier} S. Hermiller and J. Meier, Algorithms and geometry
for graph products of groups, J. Alg. 171 (1995), 230--257.
\bibitem{SLartin} D.F. Holt and S. Rees, Artin groups of large type are
automatic with regular geodesics, 
Proc. London Math. Soc 104 (2012) 486--512.
\bibitem{SLartin2} D.\,F. Holt and S. Rees, Shortlex automaticity and geodesic
regularity in Artin groups, Groups, Complex. Cryptol. 5 (2013) 1--23.
\bibitem{HRRbook} D.\,F. Holt, S. Rees and C.\,E. R\"over,
{\em Groups, languages and automata}, LMS student texts 88, CUP, 2017.
\bibitem{HuangOsajda} J. Huang and D. Osajda, Large-type Artin groups are
systolic, {\tt arXiv:1706.05473}
\bibitem{JS} T. Januszkiewicz and J. Swiatkowski, Simplicial nonpositive
curvature, Publ. Math. Inst. Hautes Etudes Sci. No. 104 (2006), 1--85.
\bibitem{KBMAG}
D.\,F. Holt,
{\it {\sf KBMAG}---Knuth-Bendix in Monoids and Automatic Groups}, software
package (1995), available from
\verb+(http://homepages.warwick.ac.uk/~mareg/download/kbmag2/)+
\bibitem{NibloReeves_biaut} G.\,A. Niblo and L.\,D. Reeves, The geometry of cube
complexes and the complexity of their fundamental groups, Topology 37 (1998)
621--633.
\bibitem{NibloReeves_coxeter} G.\,A. Niblo and L.\,D. Reeves, Coxeter groups act
on CAT(0) cube complexes, J. Group Theory 6 (2003), no. 3, 399--413.
\bibitem{Peifer} D. Peifer, Artin groups of extra-large type are biautomatic,
J. Pure Appl. Alg. 110 (1996) 15--56.
\bibitem{Swiatkowski} J. Swiatkowski, Regular path systems and (bi)automatic groups, Geometriae Dedicata 118 (2006) 23--48.
\bibitem{VanWyk} L. VanWyk, Graph groups are biautomatic, J. Pure App. Alg. 94 ](1994), 341--352.
\end{thebibliography}
\end{document}